\documentclass[3p]{elsarticle}
\usepackage{graphicx}
\usepackage{amsmath}
\usepackage{amsfonts}
\usepackage{amssymb}
\usepackage{array}
\usepackage{color}
\usepackage[colorlinks=true,linkcolor=blue,urlcolor=blue,citecolor=blue]{hyperref}
\usepackage{mathtools}
\usepackage{nccmath}
\usepackage[normalem]{ulem}
\usepackage{booktabs}
\usepackage{blkarray}
\usepackage{tikz}

\newtheorem{theorem}{Theorem}
\newtheorem{proposition}{Proposition}

\newdefinition{definition}{Definition}
\newdefinition{example}{Example}

\newproof{proof}{Proof}
\newproof{sketch}{Sketch of the proof}

\newcommand{\floorfrac}[2]{\left\lfloor\frac{#1}{#2}\right\rfloor}
\newcommand{\ceilfrac}[2]{\left\lceil\frac{#1}{#2}\right\rceil}

\journal{Applied Mathematics and Computation}

\begin{document}

\begin{frontmatter}

\title{On interval transmission irregular graphs\tnoteref{c}}
\tnotetext[c]{This work was supported and funded by Kuwait University Research Grant No.\ SM04/19. 
              \newline\indent\indent
              Corresponding author: Salem Al-Yakoob.}

\author[KU]{Salem Al-Yakoob}
\ead{smalyakoob@gmail.com}

\author[MISANU]{Dragan Stevanovi\'c}
\ead{dragance106@yahoo.com}

\address[KU]{Department of Mathematics, Faculty of Science, Kuwait University, Safat 13060, Kuwait}

\address[MISANU]{Mathematical Institute, Serbian Academy of Sciences and Arts, 
              Kneza Mihaila 36, 11000 Belgrade, Serbia}

\begin{abstract}
Transmission of a vertex~$v$ of a connected graph~$G$ is 
the sum of distances from~$v$ to all other vertices in~$G$.
Graph $G$ is transmission irregular (TI) 
if no two of its vertices have the same transmission,
and $G$ is interval transmission irregular (ITI) 
if it is TI and the vertex transmissions of $G$ form a sequence of consecutive integers.
Here we give a positive answer to the question of Dobrynin [Appl Math Comput 340 (2019), 1--4]
of whether infinite families of ITI graphs exist.
\end{abstract}


\end{frontmatter}

\section{Introduction}

For a simple graph $G$, we denote by $V_G$ and $E_G$ the sets of its vertices and edges, respectively.
Degree $\deg_G(u)$ of vertex~$u\in V_G$ is defined as the number of edges in~$E_G$ that are incident to~$u$.
A walk of length~$k$ between two vertices $u$ and~$v$ of~$G$ is 
a sequence of its vertices $W\colon u=w_0,w_1,\dots,w_k=v$
such that $w_i$ and $w_{i+1}$ are adjacent for each $i=0,\dots,k-1$.
Graph $G$ is connected if there exists a walk between each pair of vertices of~$G$.
The distance $d_G(u,v)$ between vertices $u$ and~$v$ of connected graph~$G$ is then
the length of the shortest walk between $u$ and~$v$,
while the transmission of vertex~$u$ is 
the sum of distances from~$u$ to all other vertices of~$G$:
$$
Tr_G(u)=\sum_{v\in V_G} d_G(u,v).
$$

A graph is {\em transmission irregular} (TI) if no two of its vertices have the same transmission.
Recent interest in TI graphs was motivated by the observation of Alizadeh and Klav\v zar~\cite{alkl} 
that TI graphs have maximal Wiener index complexity,
where the complexity of a summation-type topological index is defined
as the number of its distinct constituent summands~\cite{aaks,alkl2,kjrmp}.
Alizadeh and Klav\v zar~\cite{alkl} also noted that TI graphs are rather rare:
on one hand, almost all graphs have diameter two~\cite{blha},
while on the other hand, 
transmissions of vertices in a graph~$G$ with diameter two are 
directly related to their degree through $Tr_G(u)=2(|V_G|-1)-\mathop{deg}_G(u)$.
As every graph contains a pair of vertices with equal degrees,
this implies that almost all graphs are not TI.

Accordingly, this motivated researchers to explore constructions of infinite families of TI graphs.
Alizadeh and Klav\v zar~\cite{alkl} characterized 
TI starlike trees with three branches, of which one branch has length one,
while the present authors in~\cite{yast} extended this to
the characterization of all TI starlike trees with three branches.
Xu and Klav\v zar~\cite{xukl} showed that
a starlike tree whose branch lengths form a sequence of consecutive integers is TI 
if this tree has an odd number of vertices.
Moreover, in a short series of papers,
Dobrynin constructed infinite families of
TI trees of even order \cite{dobr1}, 2-connected TI graphs~\cite{dobr2,dobr0} and 3-connected cubic TI graphs~\cite{dobr3}.

Observing that vertex transmissions in 
one of the examples of small 2-connected TI graphs found in~\cite{dobr2} are consecutive,
Dobrynin introduced therein a particular subclass of TI graphs characterized as follows:
a graph is {\em interval transmission irregular} (ITI) 
if it is TI and 
the set of its vertex transmissions is equal to $[a,b]\cap\mathbb{Z}$ for some $a,b\in\mathbb{Z}$.
Dobrynin then proposed in both~\cite{dobr2,dobr0} the question of whether there exist infinite families of ITI graphs,
and our goal in this research effort is to give an affirmative answer to this question.

All ITI graphs with up to ten vertices are shown in Fig.~\ref{fig-small-ITI}.
Initially, our search for an infinite ITI family was focused
on the fourth graph from left in the second row of this figure,
because this is the only graph in Fig.~\ref{fig-small-ITI} without any pendent vertex and
this graph motivated Dobrynin in~\cite{dobr2} to define the concept of ITI graphs.
This particular graph suggests that adding paths between the vertices of a small core may yield new TI graphs,
and this idea is explored in Section~\ref{sc-chordal-paths},
with pertinent findings presented therein.
Computational search unearthed many ITI graphs of this form.
Note that our theoretical studies led to several new infinite families of TI graphs 
(as presented in Section~\ref{sc-chordal-paths}),
however it turned out that these new families contain ITI graphs only sporadically,
leaving the multitude of ITI graphs of this form unexplained.
A particular result from this section, that may be of interest in its own right, is
the fact that transmissions of vertices lying on an internal path form a unimodal sequence
(Theorem~\ref{th-unimodal-transmissions}).
Another observation that is worth stating is that 
the Cartesian product of two TI graphs
with relatively prime numbers of vertices is again TI,
provided that at least one of the factors is actually a {\em modulo transmission irregular} (MTI) graph,
where an MTI graph is defined as a graph in which 
no two vertices have the same transmission modulo the number of its vertices.
This is proved in Theorem \ref{th-product-with-mti},
which is the subject of Section~\ref{sc-cartesian-product}.
At the end, by studying numerous examples of ITI graphs on 11 vertices,
we noticed that a sizable number of these examples have diameter three
and that they usually have a pair of vertices of large degrees that differ by one.
Finally, taking this general structure as a starting point,
led us relatively quickly to an instance of an infinite family of ITI graphs
that is described in Section~\ref{sc-infinite-iti},
thus affirmatively settling Dobrynin's question raised in~\cite{dobr2}.

\begin{figure}[ht!]
\begin{center}
\includegraphics[width=\textwidth]{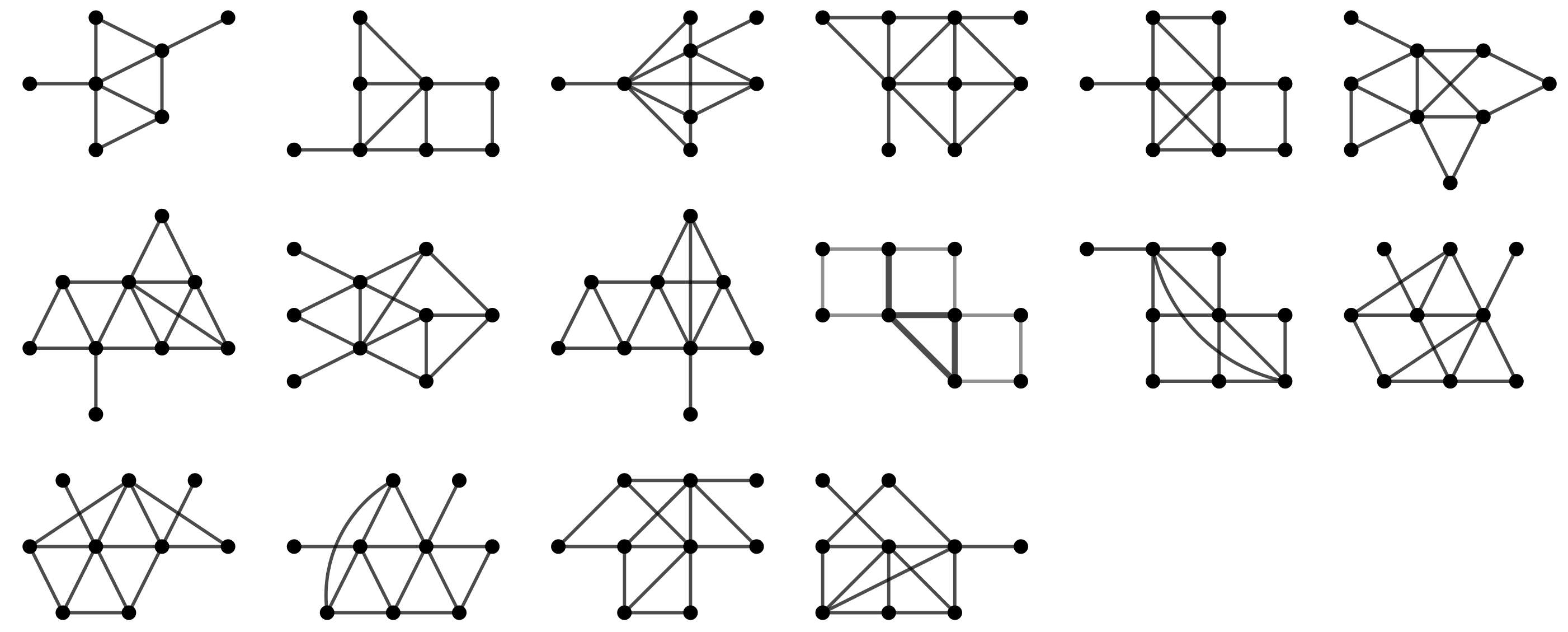}
\end{center}
\caption{ITI graphs with up to ten vertices.}
\label{fig-small-ITI}
\end{figure}

\section{Distances in graphs with added chordal paths}
\label{sc-chordal-paths}

Results of preliminary computational experiments led us to observe that 
a number of classes of TI graphs can be found
when some vertices of the underlying core are joined by paths of different lengths.
This was also the case with an infinite family of 2-connected TI graphs
identified by Dobrynin in~\cite{dobr1}.
To simplify theoretical treatment of transmissions in such graphs,
we first consider how to calculate distances between their vertices.

\begin{definition}
Let $G$ be a given graph.
For some $k\geq 0$,
let $\mathcal{A}=\{P_1,\dots,P_k\}$
where each~$P_i$ is a triplet $(u_i,v_i,s_i)$ 
consisting of two vertices $u_i,v_i\in V_G$ and a nonnegative integer~$s_i$.
The graph $G+\mathcal{A}$ is obtained from~$G$ by adding to it, 
for each $1\leq i\leq k$, a new path from $u_i$ to~$v_i$ with $s_i$ internal vertices.
The graph $G$ is called the {\em core} of $G+\mathcal{A}$,
while the paths added to it are called the {\em chordal paths}.
\end{definition}

In the above definition,
chordal paths from~$\mathcal{A}$ actually represent internal paths in~$G+\mathcal{A}$.
We opted to use different terminology here in order to 
emphasize the core and the addition of new paths to it,
instead of just considering these paths as parts of the final graph.
Fig.~\ref{fig-1} illustrates the above definition,
where to the core~$G$, shown on the left-hand side,
three chordal paths $\mathcal{A}=\{P_1,P_2,P_3\}$ are added,
where $P_1=P_2=(b,d,3)$ and $P_3=(c,e,3)$,
producing the graph $G+\mathcal{A}$ shown in the middle.

\begin{figure}[ht!]
\begin{center}
\includegraphics[scale=0.7]{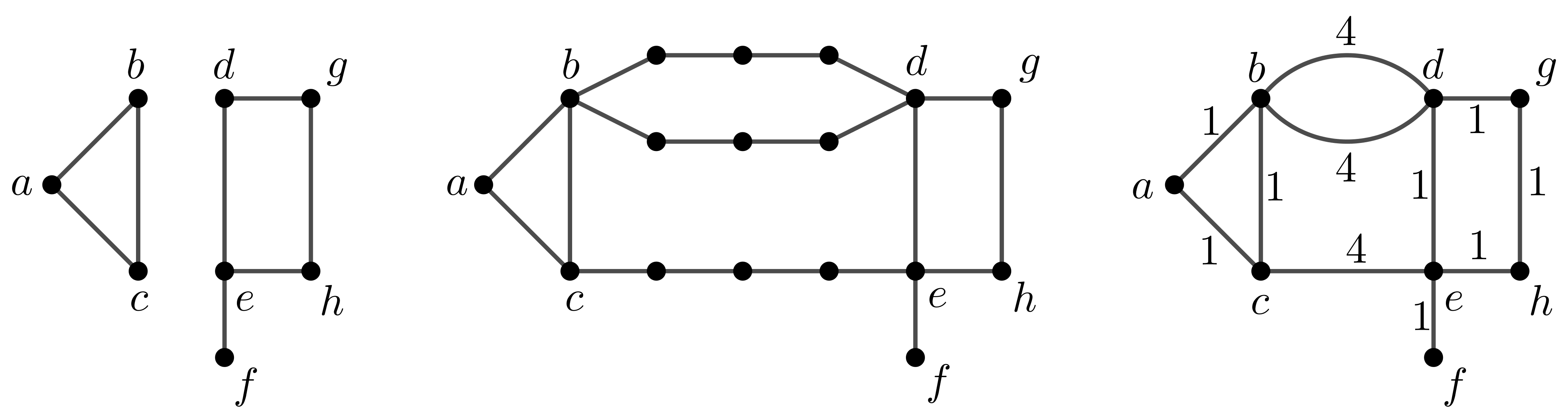}
\end{center}
\caption{The core $G$, the graph $G+\{P_1,P_2,P_3\}$ 
         and the auxiliary weigthed graph~$G'$ (see text).}
\label{fig-1}
\end{figure}

% Floyd-Warshall on the auxiliary weighted variant of the core
Assume we are given the core~$G$ and the set~$\mathcal{A}$ of chordal paths.
Each of the internal vertices of chordal paths in~$G+\mathcal{A}$ has degree two,
so that
a shortest walk in~$G+\mathcal{A}$ that connects two vertices of the core~$G$
either includes a whole chordal path or avoids it completely.
From that aspect, 
we can treat a chordal path with $s$~internal vertices between two vertices $u$ and~$v$ of~$G$ 
simply as a new edge of weight~$s+1$ (which is the length of the chordal path) between $u$ and~$v$ in~$G$.
Thus we can produce from~$G$
an auxiliary graph~$G'$ by adding weighted edges corresponding to chordal paths,
so that distances in~$G'$ are equal to distances between the core vertices in~$G+\mathcal{A}$.
For the core~$G$ shown in Fig.~\ref{fig-1},
its auxiliary graph~$G'$ is shown on the right-hand side.
Note that $G'$ has the same number of vertices as~$G$, 
unlike $G+\mathcal{A}$ which adds a number of internal vertices to~$G$,
so that applying the Floyd-Warshall algorithm to~$G'$ is more efficient than applying it to~$G+\mathcal{A}$.

% Then comes the consideration of internal vertices on chordal paths
Having determined distances between the core vertices in~$G+\mathcal{A}$ 
by using the auxiliary weighted graph~$G'$,
it is easy to find out distances that involve internal vertices of chordal paths in $G+\mathcal{A}$ as well.
Assume that $c$ is an internal vertex of the chordal path~$P=(u,v,s)$ of~$G+\mathcal{A}$,
so that $c$ is at distance~$k$ from~$u$ (and hence at distance $s+1-k$ from~$v$).
A shortest walk in~$G+\mathcal{A}$ from~$c$ to another core vertex~$b$
must first follow the chordal path~$P$ straight to either $u$ or~$v$,
after which it will follow a shortest path from either $u$ or~$v$ to~$b$ in the auxiliary graph~$G'$,
so that
\begin{equation}
\label{eq-from-chord-to-core}
d_{G+\mathcal{A}}(c,b) = \min\{k + d_{G'}(u,b), \quad s+1-k + d_{G'}(v,b)\}.
\end{equation}
If $c'$ is an internal vertex of another chordal path~$P'=(u',v',s')\neq P$,
so that $c'$ is at distance $k'$ from $u'$ (and at distance $s'+1-k'$ from $v'$),
then a shortest walk from~$c$ to~$c'$ in $G+\mathcal{A}$
follows $P$ straight to either $u$ or~$v$,
then a shortest walk from either $u$ or~$v$ to either $u'$ or~$v'$ in~$G'$,
after which it follows $P'$ straight to~$c'$.
Hence
\begin{eqnarray}
\label{eq-from-chord-to-chord-different}
d_{G+\mathcal{A}}(c,c') 
    &=& \min\{k + d_{G'}(u,u') + k', \quad k + d_{G'}(u,v') + s'+1-k', \\
\nonumber
    & & \phantom{\min\{} s+1-k + d_{G'}(v,u') + k', \quad s+1-k + d_{G'}(v,v') + s'+1-k'\}.
\end{eqnarray}
Note that if $c''$ is an internal vertex of the same chordal path~$P$ as~$c$,
so that $c''$ is at distance~$k''$ from~$u$ (and at distance $s+1-k''$ from~$v$),
then a shortest walk from $c$ to~$c''$ may, on one hand, follow~$P$ directly between these vertices,
while on the other hand,
it may follow $P$ from~$c$ to either $u$ or~$v$,
then a shortest walk from $u$ to $v$ in the auxiliary graph~$G'$,
after which it follows~$P$ again from the other side back to~$c'$.
Hence in this case
\begin{equation}
\label{eq-from-chord-to-chord-same}
d_{G+\mathcal{A}}(c,c'')=\min\{|k''-k|, \quad k + d_{G'}(u,v) + s+1-k'', \quad s+1-k + d_{G'}(v,u) + k''\}.
\end{equation}

% A few selected examples
Let us now illustrate 
the process of calculating distances and transmissions in graphs with added chordal paths
on a few selected examples that will lead to several new infinite families of TI graphs.

\begin{figure}[ht!]
\begin{center}
\includegraphics[scale=0.7]{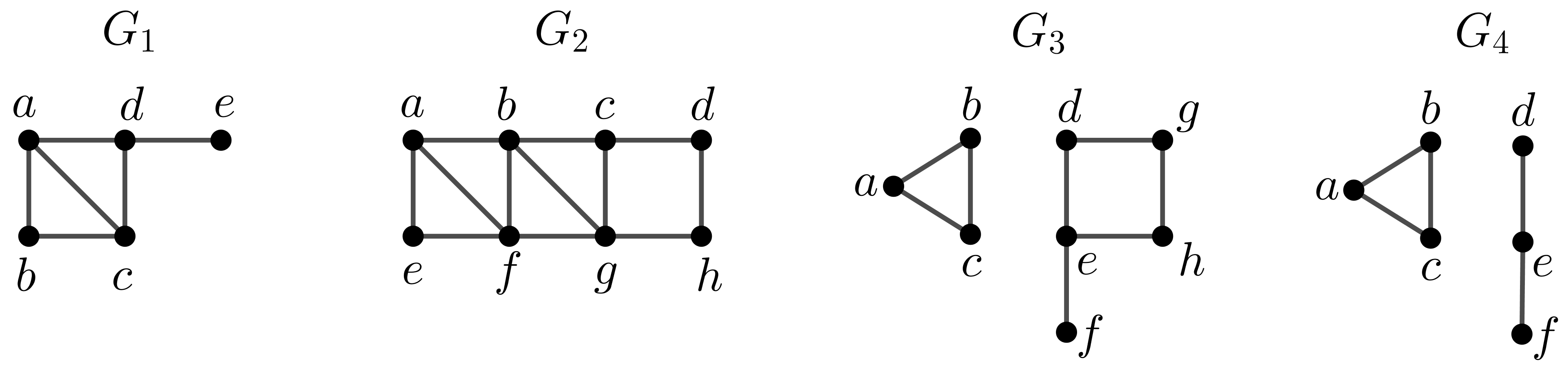}
\end{center}
\caption{The cores used in Examples \ref{ex-1}--\ref{ex-4}.}
\label{fig-2}
\end{figure}

\begin{example}
\label{ex-1}
% 4 0 2
% o-o-o
%   |/|
%   o-o
%   1 3
%
% a 2 3 6k+1 and k>=1 yields TI graph
%
Let $G_1$ be the core shown in Fig.~\ref{fig-2},
let $\mathcal{A}^1_n=\{P_1\}$ with $P_1=(a,b,n)$ for some integer $n\geq 2$,
and let $\mathcal{G}^1_n=G_1+\mathcal{A}^1_n$.
Distances in the auxiliary weighted graph $G'_1$ yield
distances between the core vertices of $\mathcal{G}^1_n$,
shown in matrix form in Fig.~\ref{fig-ex-1}.

\begin{figure}[ht!]
\begin{center}
$$
\begin{blockarray}{ccccccc}
  & a & b & c & d & e & x \\
\begin{block}{c[ccccc|c]}
a & \,0 & 1 & 1 & 1 & 2 & \min\{k,   n-k+2\} \\
b & \,1 & 0 & 1 & 2 & 3 & \min\{k+1, n-k+1\} \\
c & \,1 & 1 & 0 & 1 & 2 & \min\{k+1, n-k+2\} \\
d & \,1 & 2 & 1 & 0 & 1 & \min\{k+1, n-k+3\} \\
e & \,2 & 3 & 2 & 1 & 0 & \min\{k+2, n-k+4\} \\
\end{block}
\end{blockarray}
$$
\caption{Matrix of distances from the core vertices of~$\mathcal{G}^1_n$.}
\label{fig-ex-1}
\end{center}
\end{figure}

Now, let $x$ be an internal vertex of the chordal path~$P_1$
at distance~$k$ from~$a$ (and at distance $n+1-k$ from~$b$) for $1\leq k\leq n$.
Eq.~(\ref{eq-from-chord-to-core}) yields for any core vertex~$z$
$$
d_{\mathcal{G}^1_n}(x,z) = \min\{k + d_{G'_1}(a,z), \quad n+1-k + d_{G'_1}(b,z)\},
$$
which is shown in the last column in Fig.~\ref{fig-ex-1}.
If $x'$ is another internal vertex of~$P_1$ at distance~$k'$ from~$a$,
then from Eq.~(\ref{eq-from-chord-to-chord-same}) we have
$$
d_{\mathcal{G}^1_n}(x,x') = \min\{|k'-k|, \quad n+2+k-k', \quad n+2-k+k' \}.
$$
Summing these expressions and taking into account that
for constants $\alpha$ and~$\beta$ with $\alpha-\beta\leq n-2$ 
the following identity holds
$$
\sum_{k=1}^n \min\{k+\alpha, n-k+\beta\}
% = \floorfrac{n-\alpha+\beta}{2}^2
% -(n-\alpha+\beta-1)\floorfrac{n-\alpha+\beta}2
% = \frac{n(n+2\beta-1)}2
%  - \floorfrac{n-\alpha+\beta}2\ceilfrac{n-\alpha+\beta-2}2
 = \frac14\left(n^2 + n(2\beta+2\alpha-1) - (\beta-\alpha)(\beta-\alpha-1)\right) 
  +\frac12\floorfrac{n+\beta-\alpha}2,
$$
%%
%% n(n+p)/2 - \floorfrac{n+q}2\ceilfrac{n+q-2}2 
%% = [n^2 + n(2p-2q+1) - q(q-1)]/4 + \frac12\floorfrac{n+q}2
%%
%% p=2\beta-1, q=\beta-\alpha
%% n(n+p)/2 - \floorfrac{n+q}2\ceilfrac{n+q-2}2 
%% = [n^2 + n(2\beta+2\alpha-1) - (\beta-\alpha)(\beta-\alpha-1)]/4 + \frac12\floorfrac{n+\beta-\alpha}2
%%
we accordingly obtain the transmissions of the vertices of $\mathcal{G}^1_n$ as follows:
\begin{alignat*}{2}
Tr_{\mathcal{G}^1_n}(a) &= 5+\sum_{k=1}^n \min\{k,   n-k+2\}  
 &&=\frac{n^2+3n+18}4+\frac12\floorfrac{n+2}2, \\
% &&=5+ \frac{n(n+3)}2 -\floorfrac{n+2}2\ceilfrac{n}2, \\
% &&=5+\frac{n(n+3)}2+\left(\left\lfloor\frac n2\right\rfloor+1\right)\left(\left\lfloor\frac n2\right\rfloor-n\right), \\
Tr_{\mathcal{G}^1_n}(b) &= 7+\sum_{k=1}^n \min\{k+1, n-k+1\}  
 &&=\frac{n^2+3n+28}4+\frac12\floorfrac n2, \\
% &&=7+ \frac{n(n+1)}2 -\floorfrac n2\ceilfrac{n-2}2, \\
% &&=7+\frac{n(n+3)}2+\left\lfloor\frac n2\right\rfloor\left(\left\lfloor\frac n2\right\rfloor-n+1\right), \\
Tr_{\mathcal{G}^1_n}(c) &= 5+\sum_{k=1}^n \min\{k+1, n-k+2\}  
 &&=\frac{n^2+5n+20}4 + \frac12\floorfrac{n+1}2, \\
% &&=5 + \frac{n(n+3)}2 -\floorfrac{n+1}2\ceilfrac{n-1}2, \\
% &&=5+\frac{n(n+3)}2+\left\lfloor\frac{n+1}2\right\rfloor\left(\left\lfloor\frac{n+1}2\right\rfloor-n\right), \\
Tr_{\mathcal{G}^1_n}(d) &= 5+\sum_{k=1}^n \min\{k+1, n-k+3\}  
 &&=\frac{n^2 + 7n + 18}4 + \frac12\floorfrac{n+2}2, \\      
% &&=5 + \frac{n(n+5)}2 -\floorfrac{n+2}2\ceilfrac n2, \\
% &&=5+\frac{n(n+3)}2+\left\lfloor\frac n2\right\rfloor\left(\left\lfloor\frac n2\right\rfloor-n+1\right), \\
Tr_{\mathcal{G}^1_n}(e) &= 8+\sum_{k=1}^n \min\{k+2, n-k+4\}  
 &&=\frac{n^2 + 11n + 30}4 + \frac12\floorfrac{n+2}2,
% &&=8 + \frac{n(n+7)}2 -\floorfrac{n+2}2\ceilfrac n2,
% &&=8+\frac{n(n+5)}2+\left\lfloor\frac n2\right\rfloor\left(\left\lfloor\frac n2\right\rfloor-n+1\right),
\end{alignat*}
%
%n=6m+1
%z=\lfloor n/2\rfloor=3m
%k  <=n-k+2  =>  2k<=n+2=6m+3  => k<=3m+1
%k+1<=n-k+1  =>  2k<=n  =6m+1  => k<=3m
%k+1<=n-k+2  =>  2k<=n+1=6m+2  => k<=3m+1
%k+1<=n-k+3  =>  2k<=n+2=6m+3  => k<=3m+1
%k+2<=n-k+4  =>  2k<=n+2=6m+3  => k<=3m+1
%
%Transmission from x to core vertices:
% k<=3m
%  a-x: k
%  b-x: k+1
%  c-x: k+1
%  d-x: k+1
%  e-x: k+2
%  sum: 5k+5
% k=3m+1
%  a-x: k
%  b-x: n-k+1
%  c-x: k+1
%  d-x: k+1
%  e-x: k+2
%  sum: n+3k+5
% k>=3m+2
%  a-x: n-k+2
%  b-x: n-k+1
%  c-x: n-k+2
%  d-x: n-k+3
%  e-x: n-k+4
%  sum: 5n-5k+12
% total: [k=1 to 3m] + [k=3m+1] + [k=3m+2 to 6m+1]
%  a-x: [sum k]  +[3m+1]+[sum n-k+2]=3m(3m+1)/2+3m+1+3m(3m+3)/2=3m(3m+2)+3m+1=9m^2+9m+1
%  b-x: [sum k+1]+[3m+1]+[sum n-k+1]=3m(3m+3)/2+3m+1+3m(3m+1)/2=3m(3m+2)+3m+1=9m^2+9m+1
%  c-x: [sum k+1]+[3m+2]+[sum n-k+2]=3m(3m+3)/2+3m+2+3m(3m+3)/2=3m(3m+3)+3m+2=9m^2+12m+2
%  d-x: [sum k+1]+[3m+2]+[sum n-k+3]=3m(3m+3)/2+3m+2+3m(3m+5)/2=3m(3m+4)+3m+2=9m^2+15m+2
%  e-x: [sum k+2]+[3m+3]+[sum n-k+4]=3m(3m+5)/2+3m+3+3m(3m+7)/2=3m(3m+6)+3m+3=9m^2+21m+3
%
while
\begin{eqnarray*}
Tr_{\mathcal{G}^1_n}(x) 
 &=& \min\{k,   n-k+2\} + \min\{k+1, n-k+1\} + \min\{k+1, n-k+2\} + \min\{k+1, n-k+3\} \\
 &+& \min\{k+2, n-k+4\} + \sum_{k'=1}^n \min\{|k'-k|, n+2+k-k', n+2-k+k'\}.
\end{eqnarray*}
For a particular choice $n=6m+1$ we obtain the following transmissions,
listed here in decreasing order:
%
%n=6m+1
%\lfloor n/2\rfloor = 3m
%\lfloor{n+1}/2\rfloor = 3m+1
%
\begin{eqnarray*}
Tr_{\mathcal{G}^1_{6m+1}}(e) &=& 9m^2+21m+11, \\
Tr_{\mathcal{G}^1_{6m+1}}(d) &=& 9m^2+15m+ 7, \\
Tr_{\mathcal{G}^1_{6m+1}}(c) &=& 9m^2+12m+ 7, \\
Tr_{\mathcal{G}^1_{6m+1}}(b) &=& 9m^2+ 9m+ 8, \\
Tr_{\mathcal{G}^1_{6m+1}}(a) &=& 9m^2+ 9m+ 6,
\end{eqnarray*}
and
\begin{eqnarray*}
Tr_{\mathcal{G}^1_{6m+1}}(x) &=& \left\{\begin{array}{rl}
9m^2 + 9m + 6 + 3k, & \mbox{if } 1\leq k\leq 3m+1, \\
9m^2 +27m +14 - 3k, & \mbox{if } 3m+2\leq k\leq 6m+1.
\end{array}\right.
\end{eqnarray*}
It is apparent that $Tr_{\mathcal{G}^1_{6m+1}}(x)$, for $1\leq k\leq 3m+1$, is larger than $Tr_{\mathcal{G}^1_{6m+1}}(a)$ 
and, since $Tr_{\mathcal{G}^1_{6m+1}}(x)$ is divisible by three, it is different from all other transmissions.
Similarly, $Tr_{\mathcal{G}^1_n}(x)$ for $3m+2\leq k\leq 6m+1$ is congruent to~2 modulo~3,
and as a result, $Tr_{\mathcal{G}^1_{6m+1}}(x)$ is equal to either $Tr_{\mathcal{G}^1_n}(e)$ or $Tr_{\mathcal{G}^1_n}(b)$.
However, the equality $Tr_{\mathcal{G}^1_n}(x)=Tr_{\mathcal{G}^1_n}(e)$ would imply $k=2m+1$,
while the equality $Tr_{\mathcal{G}^1_n}(x)=Tr_{\mathcal{G}^1_n}(b)$ leads to $k=6m+2$,
contradictory to the assumed range of~$k$ 
% $3m+2\leq k\leq 6m+1$ 
for $m\geq 1$.
Hence we have the following proposition.
\begin{proposition}
\label{pr-1}
For $m\geq 1$ the graph $\mathcal{G}^1_{6m+1}$ is transmission irregular.
\end{proposition}
Note that for the values of~$n$ that are not congruent to~1 modulo~6,
there always appear one or two pairs of vertices with equal transmissions in~$\mathcal{G}^1_n$.
\end{example}

\begin{example}
\label{ex-2}
%0 1 2 3 
%o-o-o-o
%|\|\| |
%o-o-o-o
%4 5 6 7
%
%a 4 7 4k+1 and k>=1 yields TI graph
%ITI for k=1.
Let $G_2$ be the core shown in Fig.~\ref{fig-2},
let $\mathcal{A}^2_n=\{P_1\}$ with $P_1=(e,h,n)$ for $n\geq 3$,
and let $\mathcal{G}^2_n=G_2+\mathcal{A}^2_n$.
Using the auxiliary weighted graph~$G'_2$ and 
Eqs.~(\ref{eq-from-chord-to-core})--(\ref{eq-from-chord-to-chord-same}),
in analogous manner as in Example~\ref{ex-1},
we can obtain expressions for transmissions of the core vertices of $\mathcal{G}^2_n$:
\begin{eqnarray*}
Tr_{\mathcal{G}^2_n}(d) &=& \frac{n^2 + 11n + 58}4 + \frac12\floorfrac{n-2}2, \\
                          % 16+\frac{n(n+3)}2-\floorfrac{n-2}2\ceilfrac{n-4}2, \\
Tr_{\mathcal{G}^2_n}(c) &=& \frac{n^2 + 11n + 48}4 + \frac12\floorfrac n2, \\
                          % 12+\frac{n(n+5)}2-\floorfrac  n2  \ceilfrac{n-2}2, \\
Tr_{\mathcal{G}^2_n}(a) &=& \frac{n^2 + 9n + 46}4 + \frac12\floorfrac{n+3}2, \\
                          % 13+\frac{n(n+7)}2-\floorfrac{n+3}2\ceilfrac{n+1}2, \\
Tr_{\mathcal{G}^2_n}(b) &=& \frac{n^2 + 9n + 40}4 + \frac12\floorfrac{n+1}2, \\
                          % 10+\frac{n(n+5)}2-\floorfrac{n+1}2\ceilfrac{n-1}2, \\
Tr_{\mathcal{G}^2_n}(e) &=& \frac{n^2 + 7n + 52}4 + \frac12\floorfrac{n+4}2, \\
                          % 16+\frac{n(n+7)}2-\floorfrac{n+4}2\ceilfrac{n+2}2, \\
Tr_{\mathcal{G}^2_n}(h) &=& \frac{n^2 + 7n + 50}4 + \frac12\floorfrac{n-2}2, \\
                          % 14+\frac{n(n+1)}2+\floorfrac{n-2}2\ceilfrac{n-4}2,
Tr_{\mathcal{G}^2_n}(f) &=& \frac{n^2 + 7n + 42}4 + \frac12\floorfrac{n+2}2, \\
                          % 11+\frac{n(n+5)}2-\floorfrac{n+2}2\ceilfrac  n2  , \\
Tr_{\mathcal{G}^2_n}(g) &=& \frac{n^2 + 7n + 40}4 + \frac12\floorfrac{n}2,
                          % 10+\frac{n(n+3)}2-\floorfrac  n2^2\ceilfrac{n-2}2, \\
\end{eqnarray*}
while the transmission of the vertex~$x$ of the chordal path~$P_1$ at distance~$k$ from~$e$ is given by
\begin{eqnarray*}
Tr_{\mathcal{G}^2_n}(x) &=& \min\{k+1, n-k+4\}+\min\{k+2, n-k+3\}+\min\{k+3, n-k+3\}+\min\{k+4, n-k+2\} \\
                        &+& \min\{k, n-k+4\}+\min\{k+1, n-k+3\}+\min\{k+2, n-k+2\}+\min\{k+3, n-k+1\} \\
                        &+& \sum_{k'=1}^n \min\{|k'-k|, n+4+k-k', n+4-k+k'\}.
\end{eqnarray*}
For the particular choice $n=4m+1$ with $k\geq 1$:
\begin{eqnarray*}
Tr_{\mathcal{G}^2_n}(d) &=& 4m^2+14m+17, \\
Tr_{\mathcal{G}^2_n}(c) &=& 4m^2+14m+15, \\
Tr_{\mathcal{G}^2_n}(a) &=& 4m^2+12m+15, \\
Tr_{\mathcal{G}^2_n}(b) &=& 4m^2+12m+13, \\
Tr_{\mathcal{G}^2_n}(e) &=& 4m^2+10m+16, \\
Tr_{\mathcal{G}^2_n}(h) &=& 4m^2+10m+14, \\
Tr_{\mathcal{G}^2_n}(f) &=& 4m^2+10m+13, \\
Tr_{\mathcal{G}^2_n}(g) &=& 4m^2+10m+12,
\end{eqnarray*}
and
\begin{eqnarray*}
Tr_{\mathcal{G}^2_n}(x) &=& \left\{\begin{array}{ll}
4 m^2 + 10 m + 16 + 4k, & \mbox{if }1\leq k\leq 2m-1, \\
4 m^2 + 18 m + 15, & \mbox{if }k=2m, \\
4 m^2 + 18 m + 16, & \mbox{if }k=2m+1, \\
4 m^2 + 26 m + 22 - 4k, & \mbox{if }2m+2\leq k\leq 4m+1.
\end{array}\right.
\end{eqnarray*}
It is now easy to see that the following proposition holds.
\begin{proposition}
\label{pr-2}
For $m\geq 1$ the graph $\mathcal{G}^2_{4m+1}$ is transmission irregular.
\end{proposition}
Among these graphs only $\mathcal{G}^2_5$ is ITI.
On the other hand, for the values of~$n$ that are not congruent to~1 modulo~4,
the graph $\mathcal{G}^2_n$ contains at least one pair of vertices with equal transmissions.
\end{example}

\begin{example}
\label{ex-3}
Let $G_3$ be the disconnected core shown in Fig.~\ref{fig-2},
let $\mathcal{A}^3_n=\{P_1,P_2\}$ with $P_1=(b,d,n)$ and $P_2=(c,e,n)$ for some integer $n\geq 1$,
and let $\mathcal{G}^3_n=G_3+\mathcal{A}^3_n$.
Calculating distances in the auxiliary weighted graph~$G'_3$ 
yields distances between the core vertices of~$\mathcal{G}^3_n$,
shown in matrix form in Fig.~\ref{fig-ex-3}.

Further, if $x$ is the internal vertex of the chordal path~$P_1$ 
at distance~$k$ from~$b$ (and at distance $n+1-k$ from~$d$) for $1\leq k\leq n$,
and $y$ is the internal vertex of the chordal path~$P_2$
at distance~$l$ from~$c$ (and at distance $n+1-l$ from~$e$) for $1\leq l\leq n$,
then Eq.~(\ref{eq-from-chord-to-core}) yields that for any core vertex~$z$
$$
d_{\mathcal{G}^3_n}(x,z) = \min\{k + d_{G'_3}(b,z), \quad n+1-k + d_{G'_3}(d,z)\}
$$
and
$$
d_{\mathcal{G}^3_n}(y,z) = \min\{l + d_{G'_3}(c,z), \quad n+1-l + d_{G'_3}(e,z)\}.
$$
These distances are shown in the last two columns in Fig.~\ref{fig-ex-3}.

\begin{figure}[ht!]
\begin{center}
$$
\begin{blockarray}{ccccccccccc}
  & a   & b   & c   & d   & e   & f   & g   & h   & x & y \\
\begin{block}{c[cccccccc|cc]}
a & 0   & 1   & 1   & n+2 & n+2 & n+3 & n+3 & n+3 & k+1   & l+1   \\
b & 1   & 0   & 1   & n+1 & n+2 & n+3 & n+2 & n+3 & k     & l+1   \\
c & 1   & 1   & 0   & n+2 & n+1 & n+2 & n+3 & n+2 & k+1   & l     \\
d & n+2 & n+1 & n+2 & 0   & 1   & 2   & 1   & 2   & n-k+1 & n-l+2 \\
e & n+2 & n+2 & n+1 & 1   & 0   & 1   & 2   & 1   & n-k+2 & n-l+1 \\
f & n+3 & n+3 & n+2 & 2   & 1   & 0   & 3   & 2   & n-k+3 & n-l+2 \\
g & n+3 & n+2 & n+3 & 1   & 2   & 3   & 0   & 1   & n-k+2 & n-l+3 \\
h & n+3 & n+3 & n+2 & 2   & 1   & 2   & 1   & 0   & n-k+3 & n-l+2 \\
\end{block}
\end{blockarray}
$$
\caption{Matrix of distances from the core vertices of~$\mathcal{G}^3_n$.}
\label{fig-ex-3}
\end{center}
\end{figure}
From Eq.~(\ref{eq-from-chord-to-chord-different}) we obtain
$$
d_{\mathcal{G}^3_n}(x,y) = \min\{k+l+1, \quad 2n+3-k-l\}.
$$
Finally, 
let $x'$ be another internal vertex of the chordal path~$P_1$ 
at distance~$k'$ from~$b$ for $1\leq k'\leq n$,
and let $y'$ be the internal vertex of the chordal path~$P_2$
at distance~$l'$ from~$c$ for $1\leq l'\leq n$.
Then from Eq.~(\ref{eq-from-chord-to-chord-same}) we have
$$
d_{\mathcal{G}^3_n}(x,x')=|k'-k|
\quad\mbox{and}\quad
d_{\mathcal{G}^3_n}(y,y')=|l'-l|.
$$
Summing these expressions 
we obtain transmissions of the vertices of~$\mathcal{G}^3_n$,
listed here in decreasing order:
\begin{alignat*}{2}
Tr_{\mathcal{G}^3_n}(a) &= 5n+15+\sum_{k=1}^n (k+1)  +\sum_{l=1}^n (l+1)   &&= n^2+8n+15, \\
Tr_{\mathcal{G}^3_n}(f) &= 3n+16+\sum_{k=1}^n (n-k+3)+\sum_{l=1}^n (n-l+2) &&= n^2+7n+16, \\
Tr_{\mathcal{G}^3_n}(g) &= 3n+15+\sum_{k=1}^n (n-k+2)+\sum_{l=1}^n (n-l+3) &&= n^2+7n+15, \\
Tr_{\mathcal{G}^3_n}(h) &= 3n+14+\sum_{k=1}^n (n-k+3)+\sum_{l=1}^n (n-l+2) &&= n^2+7n+14, \\
Tr_{\mathcal{G}^3_n}(b) &= 5n+13+\sum_{k=1}^n k      +\sum_{l=1}^n (l+1)   &&= n^2+7n+13,
\end{alignat*}
\begin{alignat*}{2}
Tr_{\mathcal{G}^3_n}(c) &= 5n+12+\sum_{k=1}^n (k+1)  +\sum_{l=1}^n l       &&= n^2+7n+12, \\
Tr_{\mathcal{G}^3_n}(x) &= 5n-2k+13+\sum_{k'=1}^n|k'-k|+\sum_{l=1}^n\min\{k+l+1,\ 2n+3-k-l\} &&=n^2+7n+13-2k, \\
Tr_{\mathcal{G}^3_n}(y) &= 5n-2l+12+\sum_{k=1}^n\min\{k+l+1,\ 2n+3-k-l\}+\sum_{l'=1}^n|l'-l| &&=n^2+7n+12-2l, \\
Tr_{\mathcal{G}^3_n}(d) &= 3n+11+\sum_{k=1}^n (n-k+1)+\sum_{l=1}^n (n-l+2) &&= n^2+5n+11, \\
Tr_{\mathcal{G}^3_n}(e) &= 3n+10+\sum_{k=1}^n (n-k+2)+\sum_{l=1}^n (n-l+1) &&= n^2+5n+10.
\end{alignat*}
Note that the values of $k$ and~$l$ range from $1$ to~$n$,
so that the value of $Tr_{\mathcal{G}^3_n}(x)$ ranges from $n^2+5n+13$ to $n^2+7n+11$,
while the value of $Tr_{\mathcal{G}^3_n}(y)$ ranges from $n^2+5n+12$ to $n^2+7n+10$,
with each transmission taking every second value in its range.
%Altogether,
%for $n\geq 2$ the graph $\mathcal{G}_n$ is transmission irregular
%with its set of transmissions equal to
%$$
%([n^2+5n+10, n^2+7n+16]\cap\mathbb Z)\cup\{n^2+8n+15\}.
%$$
We can conclude this example with the following proposition.
\begin{proposition}
\label{pr-3}
For $n\geq 2$ the graph $\mathcal{G}^3_n$ is transmission irregular,
with its set of transmissions equal to
$$
\{n^2+5n+10, \dots, n^2+7n+16\}\cup\{n^2+8n+15\}.
$$
\end{proposition}
Note that the graph $\mathcal{G}^3_2$ is also ITI,
while for $n\geq 3$ only the largest transmission $n^2+8n+15$ 
does not belong to the interval formed by the remaining transmissions.
\end{example}

\begin{example}
\label{ex-4}
Let $G_4$ be the disconnected core shown in Fig.~\ref{fig-2},
let $\mathcal{A}^4_{n,m}=\{P_1,P_2,P_3\}$ 
with $P_1=(b,d,n)$, $P_2=(c,e,n)$ and $P_3=(d,e,m)$
for some positive integers $n$ and $m$,
and let $\mathcal{G}^4_{n,m}=G_4+\mathcal{A}^4_{n,m}$.
This example generalises the previous one as $\mathcal{G}^3_n=\mathcal{G}^4_{n,2}$.

The distances in the auxiliary weighted graph $G'_4$ yield
the distances between the core vertices of $\mathcal{G}^4_{n,m}$,
which are shown in the matrix form in Fig.~\ref{fig-ex-4-core}.
Let $x$ and $x'$ be the internal vertices of the chordal path~$P_1$
at distances $p$ and~$p'$ from~$b$, respectively,
let $y$ and $y'$ be the internal vertices of the chordal path~$P_2$
at distances $q$ and~$q'$ from~$c$, respectively, and 
let $z$ and $z'$ be the internal vertices of the chordal path~$P_3$
at distances $r$ and~$r'$ from~$d$, respectively.
Eq.~(\ref{eq-from-chord-to-core}) yields the distances between the core and the chordal path vertices,
which are shown in the last three columns of Fig.~\ref{fig-ex-4-core},
while Eqs. (\ref{eq-from-chord-to-chord-different}) and~(\ref{eq-from-chord-to-chord-same}) yield
the distances between the chordal path vertices,
which are given by the following equations:
\begin{eqnarray*}
d_{\mathcal{G}^4_{n,m}}(x,x') &=& |p-p'|, \\
d_{\mathcal{G}^4_{n,m}}(x,y)  &=& \min\{p+q+1, 2n-p-q+3\}, \\
d_{\mathcal{G}^4_{n,m}}(x,z)  &=& \min\{n-p+r+1, n+m-p-r+3\}, \\
d_{\mathcal{G}^4_{n,m}}(y,y') &=& |q-q'|, \\
d_{\mathcal{G}^4_{n,m}}(y,z)  &=& \min\{n-q+r+2, n+m-q-r+2\}, \\
d_{\mathcal{G}^4_{n,m}}(z,z') &=& \min\{|r-r'|, m+r-r'+2, m-r+r'+2\}. \\
\end{eqnarray*}

\begin{figure}[ht!]
\begin{center}
$$
\begin{blockarray}{cccccccccc}
    & a       & b       & c       & d       & e       & f   
         & x             & y             & z \\
\begin{block}{c[cccccc|ccc]}
a\! & 0       & 1       & 1       & n\!+\!2 & n\!+\!2 & n\!+\!3 
         & p\!+\!1       & q\!+\!1       & \min\{n\!+\!r\!+\!2, n\!+\!m\!-\!r\!+\!3\} \\
b\! & 1       & 0       & 1       & n\!+\!1 & n\!+\!2 & n\!+\!3 
         & p             & q\!+\!1       & \min\{n\!+\!r\!+\!1, n\!+\!m\!-\!r\!+\!3\} \\
c\! & 1       & 1       & 0       & n\!+\!2 & n\!+\!1 & n\!+\!2 
         & p\!+\!1       & q             & \min\{n\!+\!r\!+\!2, n\!+\!m\!-\!r\!+\!2\} \\
d\! & n\!+\!2 & n\!+\!1 & n\!+\!2 & 0       & 1   & 2   
         & n\!-\!p\!+\!1 & n\!-\!q\!+\!2 & \min\{r,     m\!-\!r\!+\!2\}   \\
e\! & n\!+\!2 & n\!+\!2 & n\!+\!1 & 1       & 0   & 1   
         & n\!-\!p\!+\!2 & n\!-\!q\!+\!1 & \min\{r\!+\!1,   m\!-\!r\!+\!1\}   \\
f\! & n\!+\!3 & n\!+\!3 & n\!+\!2 & 2       & 1   & 0   
         & n\!-\!p\!+\!3 & n\!-\!q\!+\!2 & \min\{r\!+\!2,   m\!-\!r\!+\!2\}   \\
\end{block}
\end{blockarray}
$$
\caption{Matrix of distances from the core vertices of~$\mathcal{G}^4_{n,m}$.}
\label{fig-ex-4-core}
\end{center}
\end{figure}

Summing these expressions we obtain transmissions of the vertices of~$\mathcal{G}^4_{n,m}$:
\begin{eqnarray*}
Tr_{\mathcal{G}^4_{n,m}}(a) &=& n^2+6n+9 +\frac{m^2 + 9m    }4 + \frac12\floorfrac{m+1}2 + mn,   \\
                              % n^2+6n+9 +\frac{m(m+5)}2-\floorfrac{m+1}2\ceilfrac{m-1}2+mn,     \\
Tr_{\mathcal{G}^4_{n,m}}(b) &=& n^2+5n+8 +\frac{m^2 + 7m - 2}4 + \frac12\floorfrac{m+2}2 +mn,    \\
                              % n^2+5n+8 +\frac{m(m+5)}2-\floorfrac{m+2}2\ceilfrac m2   +mn,     \\
Tr_{\mathcal{G}^4_{n,m}}(x) &=& n^2+5n+8 +\frac{m^2 + 7m - 2}4 + \frac12\floorfrac{m+2}2 +mn-mp, \\
Tr_{\mathcal{G}^4_{n,m}}(d) &=& n^2+5n+8 +\frac{m^2 + 3m - 2}4 + \frac12\floorfrac{m+2}2,        \\
                              % n^2+5n+8 +\frac{m(m+3)}2-\floorfrac{m+2}2\ceilfrac m2,           \\
Tr_{\mathcal{G}^4_{n,m}}(c) &=& n^2+5n+7 +\frac{m^2 + 7m - 2}4 + \frac12\floorfrac{m+2}2 +mn,    \\
Tr_{\mathcal{G}^4_{n,m}}(y) &=& n^2+5n+7 +\frac{m^2 + 7m - 2}4 + \frac12\floorfrac{m+2}2 +mn-mq, \\
Tr_{\mathcal{G}^4_{n,m}}(e) &=& n^2+5n+7 +\frac{m^2 + 3m - 2}4 + \frac12\floorfrac{m+2}2,        \\
Tr_{\mathcal{G}^4_{n,m}}(f) &=& n^2+7n+11+\frac{m^2 + 7m - 2}4 + \frac12\floorfrac{m+2}2,
\end{eqnarray*}
and 
$$
Tr_{\mathcal{G}^4_{n,m}}(z) = \left\{\begin{array}{ll}
n^2+5n+7+\left(\floorfrac m2+1\right)\left(\ceilfrac m2+1\right)+(2n+4)r, 
   & \mbox{if }r\leq\floorfrac m2, \\[2pt]
n^2+4n+5+\left(\floorfrac m2+1\right)\left(\ceilfrac m2+1\right)+(2n+4)\ceilfrac m2, 
   & \mbox{if }r=\ceilfrac m2\mbox{ and $m$ is odd}, \\[2pt]
n^2+5n+6+\left(\floorfrac m2+1\right)\left(\ceilfrac m2+1\right)+(2n+4)(m+1-r), 
   & \mbox{if }r\geq\ceilfrac m2+1.
\end{array}\right.
$$

Testing out all combinations of small values of $n$ and~$m$ ($n,m\leq 50$),
it becomes evident that for odd values of~$m\leq 50$
the graph $\mathcal{G}^4_{n,m}$ contains a pair of vertices with equal transmissions for each $n\leq 50$.
On the other hand, for even values of $m\leq 50$
the graph $\mathcal{G}^4_{n,m}$ is TI for a considerable percentage of values of $n\leq 50$.
In particular, it is not hard to prove the following proposition,
which is essentially an extension of Proposition~\ref{pr-3}.
\begin{proposition}
\label{pr-4}
If $m$ is a power of two and $m\geq 4$,
then $\mathcal{G}^4_{n,m}$ is transmission irregular
for each odd~$n$ such that $n+2\nmid m+2$.
\end{proposition}

\begin{proof}
Set equal the expressions for the transmissions of the various vertices of~$\mathcal{G}^4_{n,m}$ given above in all possible ways.
Then, all but one of them lead to a contradiction
for $m$ a power of two that is divisible by four and for odd~$n$.
The only equation among these that may have a solution is
$Tr_{\mathcal{G}^4_{n,m}}(a)=Tr_{\mathcal{G}^4_{n,m}}(z)$ for $r\leq\frac m2$ 
which is equivalent to $2r=m+1-\frac{m+2}{2(n+2)}$.
Since $\frac{m+2}2$ is odd when $m$ is a power of two,
this equation has a solution if and only if $n+2$ divides~$m+2$.
\end{proof}
\end{example}

\begin{example}
\label{ex-5}
The previous examples can be generalized even further.
Let $G_5$ be the disconnected core shown in Fig.~\ref{fig-7},
let $\mathcal{A}^5_{n,m,l}=\{P_1,P_2,P_3,P_4\}$
with $P_1=(b,d,n)$, $P_2=(c,e,n)$, $P_3=(d,f,m)$ and $P_4=(f,e,l)$,
and let $\mathcal{G}^5_{n,m,l}=G_5+\mathcal{A}^5_{n,m,l}$.
Further, let $G_{6,k}$ be the disconnected core from Fig.~\ref{fig-7} for some fixed $k\geq 1$.
For a vector of positive integers $N=(n_1,\dots,n_{k+1})$,
let $\mathcal{A}^6_{N}=\{Q_1,\dots,Q_{k+1}\}$
with $Q_i=(b_i,c_{i+1},n_i)$ for $i=1,\dots,k-1$, $Q_k=(b_k,d,n_k)$ and $Q_{k+1}=(e,c_1,n_{k+1})$.
Set $\mathcal{G}^6_{k,N}=G_{6,k}+\mathcal{A}^6_{N}$.

\begin{figure}[ht!]
\begin{center}
\includegraphics[scale=0.7]{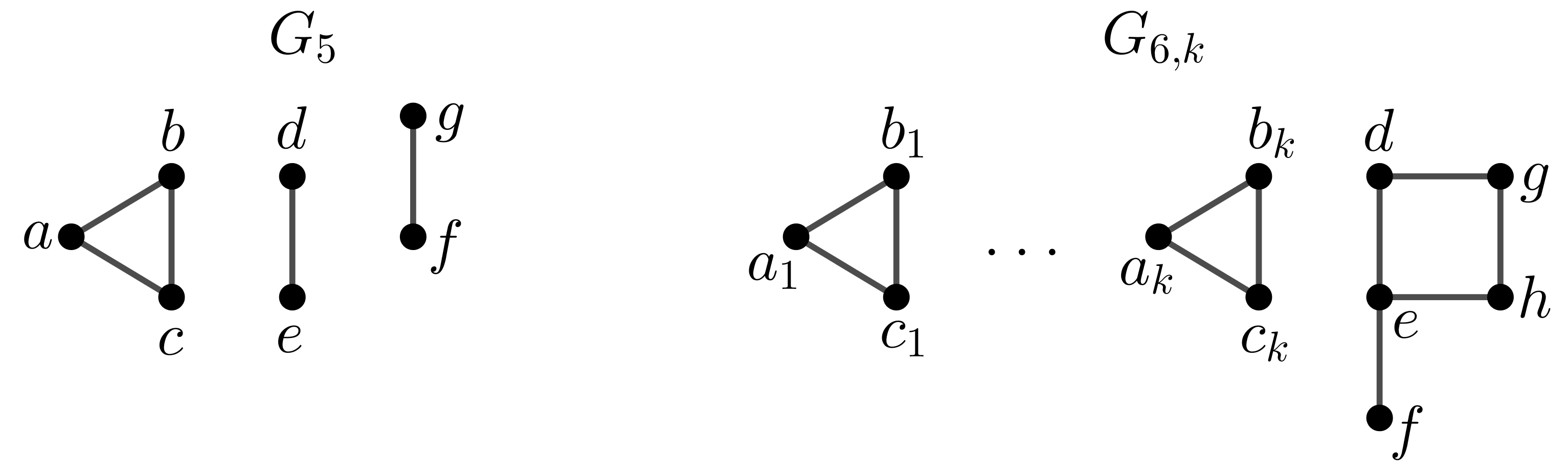}
\end{center}
\caption{The cores used in Example~\ref{ex-5}.}
\label{fig-7}
\end{figure}

In principle, one could pursue the calculation of expressions for transmissions of 
graphs $\mathcal{G}^5_{n,m,l}$ and $\mathcal{G}^6_{k,N}$
as illustrated in previous examples.
However, this is a rather tedious and time-consuming process for more complicated examples like these,
so that instead we list in Table~\ref{tb-g5-g6} small parameter values
for which graphs $\mathcal{G}^5_{n,m,l}$ and $\mathcal{G}^6_{k,N}$ are TI.

\begin{table}[ht!]
\begin{center}
\begin{tabular}{cccc@{\hskip 0.5in}ccc@{\hskip 0.5in}ccc@{\hskip 0.5in}ccc@{\hskip 0.5in}ccc}
\toprule
$\mathcal{G}^5_{n,m,l}$
   &$n$&$m$&$l$   &$n$&$m$&$l$   &$n$&$m$&$l$   &$n$&$m$&$l$   &$n$&$m$&$l$ \\
\midrule
   & 1 & 1 & 9    & 1 & 5 & 16   & 1 & 7 & 18   & 2 & 2 & 7    & 2 & 8 & 19 \\
   & 3 & 2 & 16   & 3 & 3 & 18   & 3 & 4 & 17   & 3 & 6 & 19   & 3 & 7 & 14 \\
   & 4 & 2 & 14   & 4 & 4 &  9   & 4 & 4 & 11   & 4 & 4 & 16   & 4 & 4 & 18 \\
   & 4 & 9 & 16   & 5 & 1 &  8   & 5 & 2 & 14   & 5 & 2 & 19   & 5 & 3 & 13 \\
\bottomrule
\end{tabular}

\bigskip
\begin{tabular}{ccc@{\hskip 0.5in}cc@{\hskip 0.5in}cc@{\hskip 0.5in}cc}
\toprule
$\mathcal{G}^6_{k,N}$
   & k & $N$   & k & $N$   & k & $N$   & k & $N$ \\
\midrule
   & 2 & $(1,2,1)$    & 4 & $(3,2,3,3,1)$    & 5 & $(1,13,13,1,1,1)$   & 6 & $(4, 3, 2, 4, 3, 3, 1)$ \\
   & 2 & $(18,1,1)$   & 4 & $(1,10,1,1,1)$   & 5 & $(1,13,12,3,1,2)$   & 6 & $(1, 1, 10, 6, 1, 3, 1)$ \\
   & 2 & $(25,3,3)$   & 4 & $(1,11,1,1,1)$   & 5 & $(1,13,13,1,2,2)$   & 6 & $(1, 1, 17, 1, 1, 1, 1)$ \\
   & 2 & $(31,4,4)$   & 4 & $(2,15,2,1,2)$   & 5 & $(1,13,14,1,1,2)$   & 6 & $(1, 1, 11, 6, 1, 3, 1)$ \\
\bottomrule
\end{tabular}
\end{center}
\caption{Small parameter values for which 
         graphs $\mathcal{G}^5_{n,m,l}$ and $\mathcal{G}^6_{k,N}$ are TI.}
\label{tb-g5-g6}
\end{table}
\end{example}

%                 o3=g
%                 |
%    o1=b   o4=d  o6=f
%    /|     |     
%a=0o |     |     
%    \|     |     
%    o2=c   o5=e    
% a 1 4 k
% a 2 5 k
% a 4 6 l
% a 6 5 m
%

%  3 4     7      10
%  o-o     o       o
%  | |    / \     / \
%o-o-o   o---o   o---o
%0 1 2   5   6   8   9
% a 1 5 k
% a 6 8 l
% a 9 2 m

\begin{example}
\label{ex-interactive}
To ease experimentation with transmissions of cores with added chordal paths, 
we wrote a small interactive Java program that may be downloaded from
\url{zenodo.org/record/4021916}.
This download contains both the source code and the executable file {\tt archer.jar},
so that it may be run by typing {\tt java -jar archer.jar} in the terminal,
located in the folder where the file has been downloaded.
This will start an interactive program that recognises a few simple single-letter commands:
\begin{itemize}
\item {\tt g n u1 v1 u2 v2 \dots} sets up the underlying core graph.
      For example, {\tt g 4 0 1 0 2 0 3} sets the core to have 4 vertices 
      with the edges (0,1), (0,2) and (0,3).
      Vertex numbering starts at 0;
\item {\tt g6 code} set the core through its graph6 code. 
      These codes are shortened versions of the adjacency matrix
      used by Brendan McKay's package nauty \cite{nauty};
\item {\tt a u v s} adds a new chordal path between the vertices $u$ and~$v$ with $s$ internal vertices.
\item {\tt d index} deletes the existing chordal path with given index
      (chordal path numbering also starts at~0);
\item {\tt c} clears all existing chordal paths at once;
\item {\tt x} exits the interactive program. 
\end{itemize}

The program recalculates and prints out vertex transmissions after each command, 
which makes it possible to observe their changes after additions of chordal paths.
Vertex transmissions are printed separately for each core vertex and 
along the internal vertices of each chordal path 
(counting from the chordal path end vertex that was listed first in its definition).
Vertex transmissions are then collected, sorted and printed out again in the last line
as the union of intervals and repetitions, 
so that it is easy to recognise an interval transmission integral graph, 
since the collection of the transmissions will be printed out as a single interval.
For example, 
the following commands recreate the fourth graph from the second row of Figure~\ref{fig-small-ITI}:
\begin{verbatim}
>>g 4 0 1 0 2 0 3 2 3
>>a 0 1 2
>>a 1 2 1
>>a 2 3 2 
\end{verbatim}
Program output after the last command is:
\begin{verbatim}
Vertex 0: 12
Vertex 1: 15
Vertex 2: 13
Vertex 3: 14
Arc 0 (0 1 2): 17 20
Arc 1 (1 2 1): 16
Arc 2 (2 3 2): 18 19
[12--20]
\end{verbatim}
showing that this graph is indeed ITI. 
Note that this simple interactive program is designed for quick experiments.
Hence, no syntax error checking has been implemented and 
any command that contains a typo will either be ignored (in the better case) 
or confuse the program to exit immediately (in the worse case).
\end{example}

\subsection{Computational search for ITI graphs}

Besides obtaining particular instances of ITI graphs in the previous examples,
we also ran two exhaustive computational searches
motivated by the fact that by-now-famous fourth graph from left of the second row of Fig.~\ref{fig-small-ITI} 
may also be understood as a Hamiltonian graph obtained from a cycle by adding to it a few chords.
In the first computational search we looked for ITI graphs
by adding, in all possible ways, a number of chords to a cycle of a given length.
With the exception of the cycle having 12 vertices 
for which addition of 4~particular chords leads to an ITI graph,
the addition of $(n-3)/2$ chords to the cycle with an odd number $n\geq 9$ of vertices
yielded ITI graphs in all remaining cases.
The number of ITI graphs found in this way is shown in Table~\ref{tb-cycle-with-chords},
while the drawings of such graphs with up to 15~vertices are shown in Fig.~\ref{fig-cycle-with-chords}.
Due to combinatorial explosion,
we were not able to complete the search for ITI graphs on 17~vertices with seven added chords,
so that the number of such ITI graphs is likely to be greater than~56.

\begin{table}[ht!]
\begin{center}
\begin{tabular}{c@{\hskip 0.5in}c@{\hskip 0.5in}c}
\toprule
Vertices & Chords & ITI graphs \\
\midrule
9        & 3      & 1  \\
11       & 4      & 1  \\
12       & 4      & 1  \\
13       & 5      & 6  \\
15       & 6      & 13 \\
17       & 7      & $\geq 56$ \\
\bottomrule
\end{tabular}
\end{center}
\caption{Numbers of ITI graphs found by adding a number of chords to a cycle with given number of vertices.}
\label{tb-cycle-with-chords}
\end{table}

\begin{figure}[ht!]
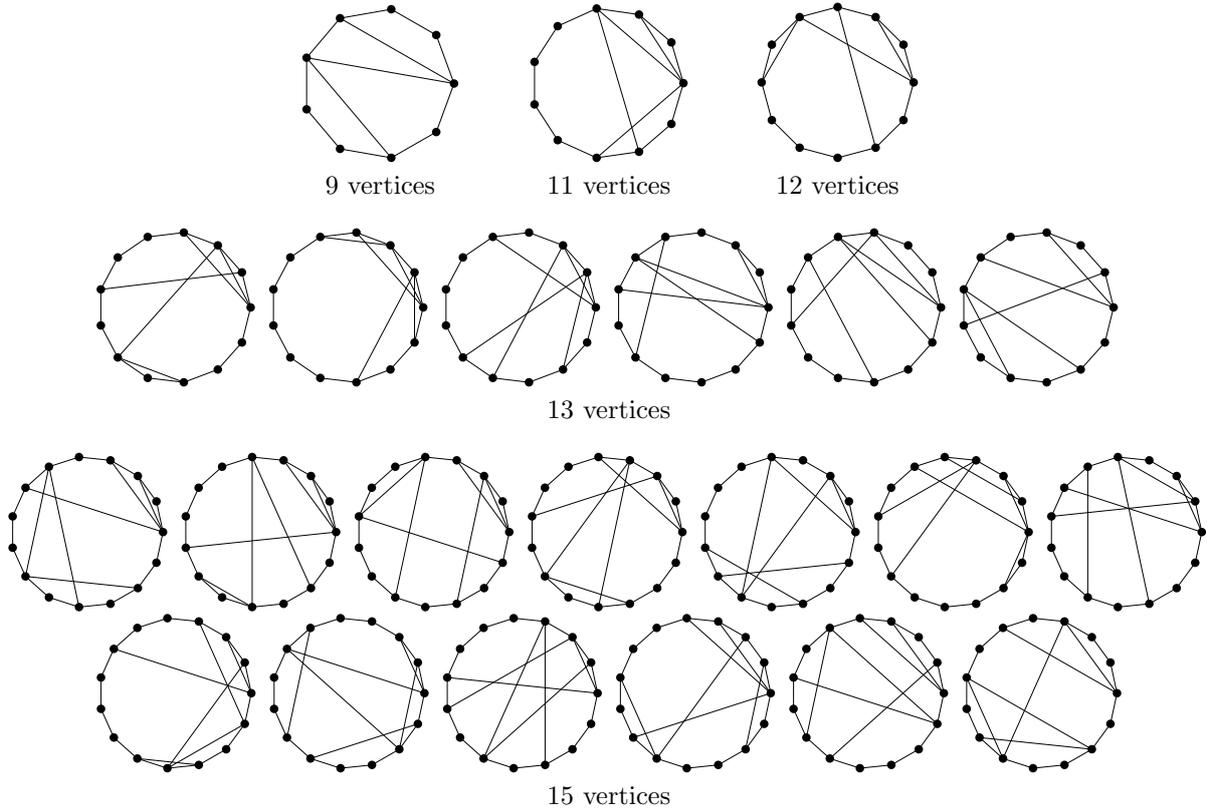

\begin{center}
\begin{minipage}[b]{1in}
\begin{center}
\scalebox{0.4}{
%%% Graph no. 1
%
}

15 vertices 
\end{center}
\end{minipage}
\end{center}
\caption{Drawings of ITI graphs with up to 15 vertices obtained by adding chords to a cycle.}
\label{fig-cycle-with-chords}
\end{figure}

We were able to find many more new examples of ITI graphs by resorting to 
a recent efficient generator of graphs with few Hamiltonian cycles,
written by Goedgebeur, Meersman and Zamfirescu~\cite{gomz}.
The use of the generator in~\cite{gomz} enabled us to exhaustively enumerate
32 ITI graphs on 15 vertices and 595 ITI graphs on 16 vertices
among graphs with a unique Hamiltonian cycle.
Among such graphs, we have further found 
87 ITI graphs on 17 vertices, 
20 ITI graphs on 19 vertices and 
5 ITI graphs on 20 vertices,
although the last three counts are not complete due to combinatorial explosion.
After examining the ITI graphs obtained in this way,
we have found those that, after deleting edges of the unique Hamiltonian cycles,
consist of (unions of) trees, (unions of) starlike trees and even (unions of) paths,
with particular instances shown in Fig.~\ref{fig-unique-Hamiltonian}.

Nevertheless, 
we were not able to observe any discernible pattern among all these examples of ITI graphs.

\begin{figure}[ht!]
\begin{center}
\includegraphics[height=3in]{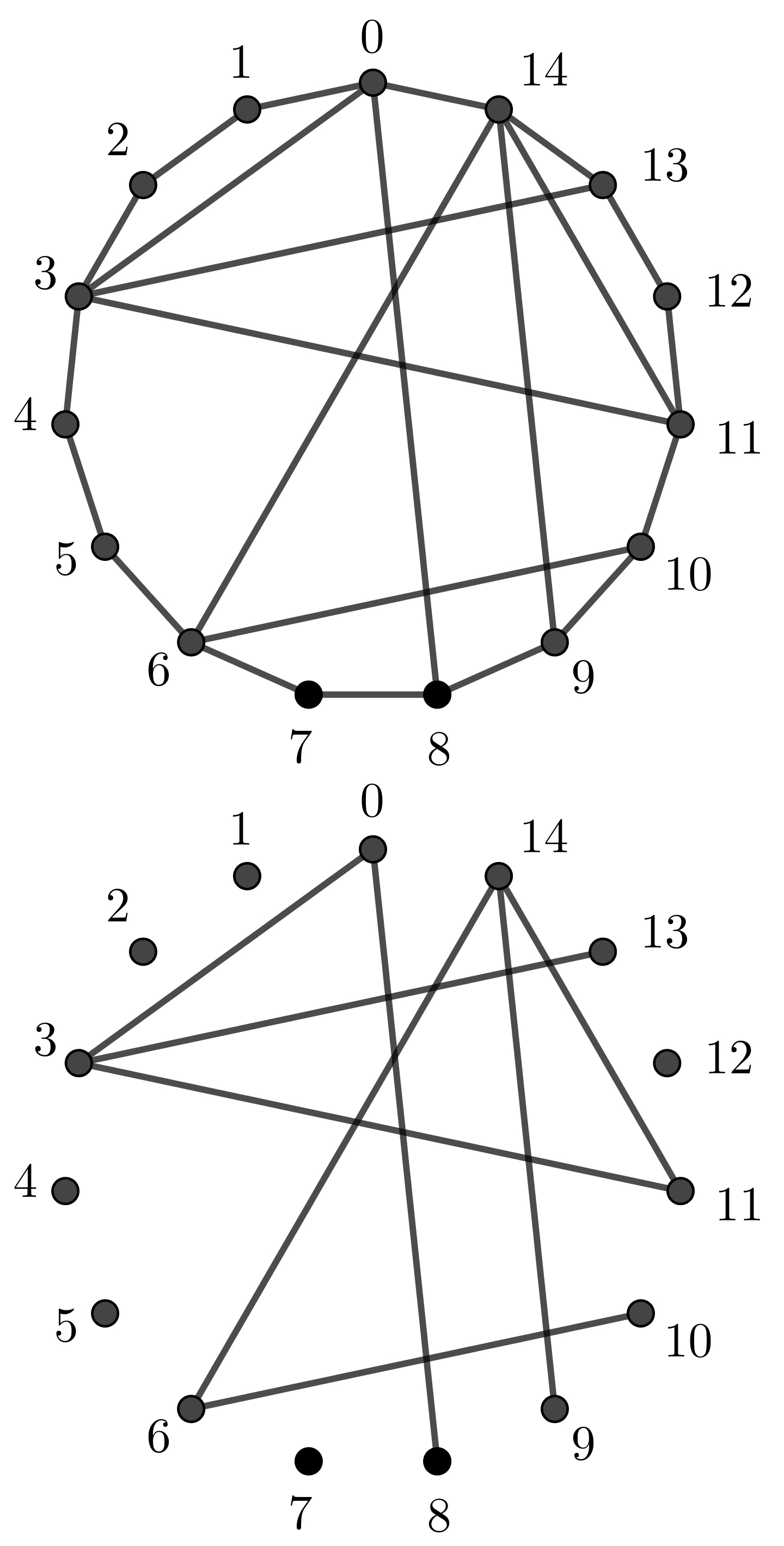}
\qquad
\includegraphics[height=3in]{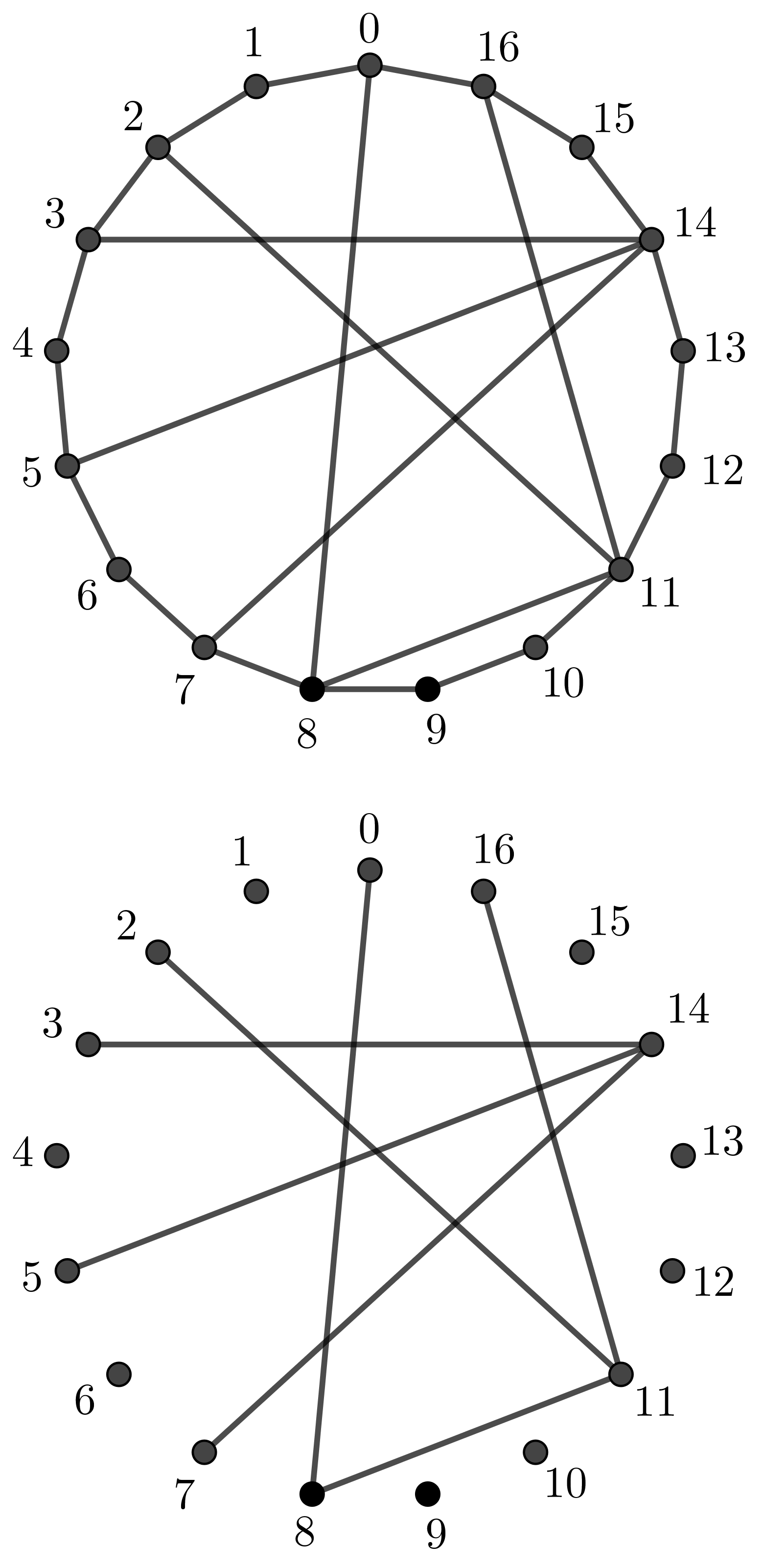}
\qquad
\includegraphics[height=3in]{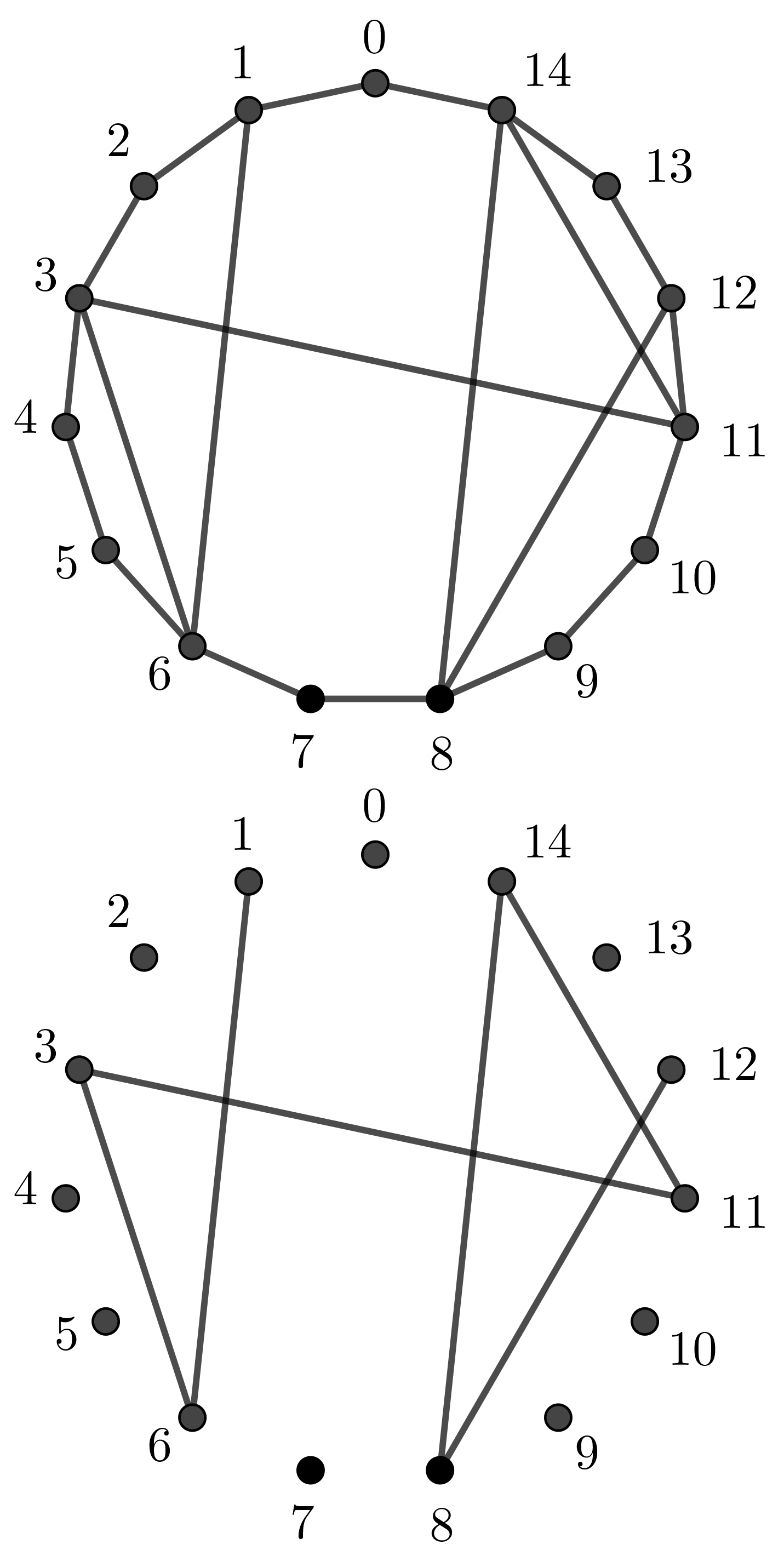}
\end{center}
\caption{Particular instances of ITI graphs with a unique Hamiltonian cycle.
Bottom drawings represent graphs after removing edges of the unique Hamiltonian cycle.}
\label{fig-unique-Hamiltonian}
\end{figure}

\subsection{Unimodality of transmissions along internal paths}

The sequence $s_1,\dots,s_k$ is {\em unimodal}
if there exists~$t$ (which can be equal to 1 or~$k$ as well) such that 
$$
s_1\leq\dots\leq s_{t-1}\leq s_t\geq s_{t+1}\geq\dots\geq s_k,
$$
while this sequence is {\em inversely unimodal} if the sequence $-s_1,\dots,-s_k$ is unimodal,
i.e., if there exists~$t$ such that
$$
s_1\geq\dots\geq s_{t-1}\geq s_t\leq s_{t+1}\leq\dots\leq s_k.
$$
During our empirical studies with the interactive program from Example~\ref{ex-interactive},
we noticed that transmissions along chordal paths are either unimodal or inversely unimodal.
Because this property holds in general, 
and not only for chordal paths attached to cores,
we state it in terms of internal paths.
Recall that an internal path in a graph~$G$ is any sequence $w_0,\dots,w_k$ of its vertices
such that $(w_i,w_{i+1})\in E_G$ for $0\leq i\leq k-1$ and 
each vertex $w_1,\dots,w_{k-1}$ has degree two.

\begin{theorem}
\label{th-unimodal-transmissions}
Let $G$ be a connected graph, $u,v\in V_G$ and 
let $u=w_0,\dots,w_k=v$ be an internal path between $u$ and~$v$ for some $k\geq 2$.
Let $H=G-w_1-\dots-w_{k-1}$ be the graph 
obtained by deleting the vertices $w_1,\dots,w_{k-1}$ from~$G$.
Then the sequence of transmissions $Tr_G(w_0), \dots, Tr_G(w_k)$ is unimodal if $H$~is connected,
and inversely unimodal if $H$~is disconnected.
\end{theorem}

\begin{figure}[ht!]
\begin{center}
\includegraphics[width=0.5\textwidth]{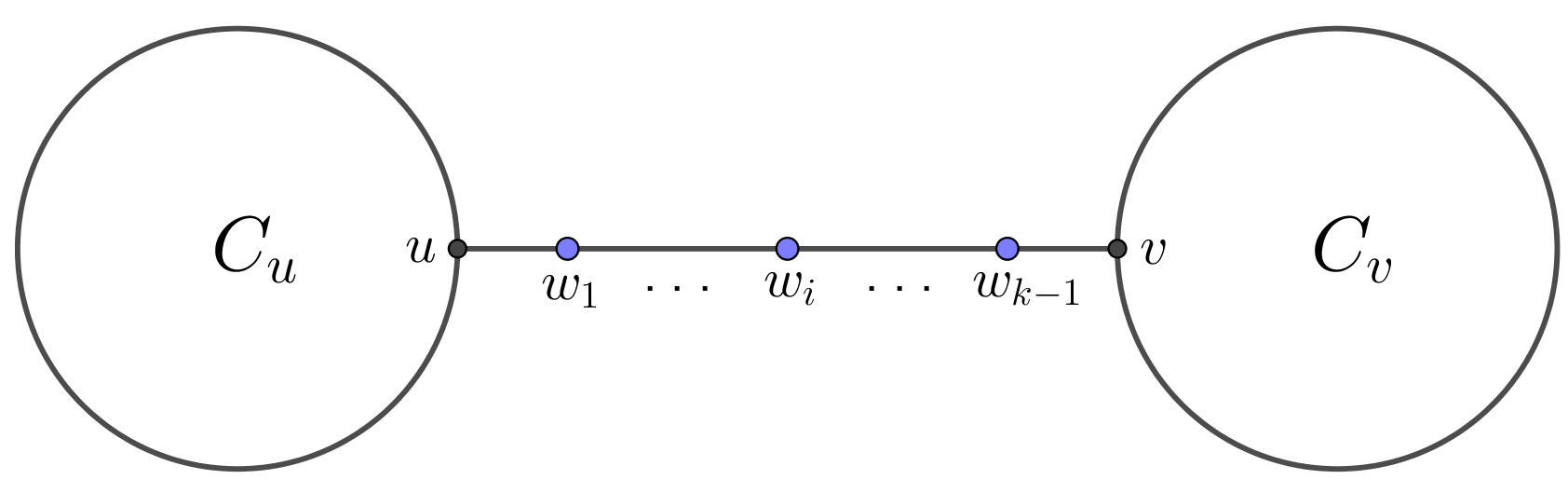}
\end{center}
\caption{The case of disconnected $H=G-w_1-\dots-w_{k-1}$.}
\label{fig-unimod-1}
\end{figure}

\begin{proof}
Suppose first that $H$ is disconnected,
so that $u$ and $v$ necessarily belong to different components of~$H$.
The edges $w_0w_1$, \dots, $w_{k-1}w_k$ are then bridges in~$G$,
and $H$ consists of two components: $C_u$ that contains~$u$ and $C_v$ that contains~$v$
(see Fig.~\ref{fig-unimod-1}).
For $0\leq i\leq k$ we have:
\begin{eqnarray*}
(\forall z\in C_u)\quad d_G(w_i,z) &=& i + d_H(u,z), \\
(\forall z\in C_v)\quad d_G(w_i,z) &=& k-i + d_H(v,z), \\
(\forall j, 0<j<k)\ d_G(w_i,w_j)&=&|i-j|. \\
% (\forall j, 1\leq j\leq k\!-\!1)\ d_G(w_i,w_j)&=&|i-j|.
\end{eqnarray*}
Hence
\begin{eqnarray*}
Tr_G(w_i) &=& i|C_u| + Tr_H(u) + (k-i)|C_v| + Tr_H(v) + \sum_{j=1}^{k-1}|i-j| \\
          &=& i^2 + i\left(|C_u|-|C_v|-k\right) + \left(Tr_H(u)+Tr_H(v)+k|C_v|+\frac{k(k-1)}2\right).
\end{eqnarray*}          
%  sum_{j=1}^{k-1} |i-j|
% =sum_{j=1}^{i-1} (i-j) + sum_{j=i}^{k-1} (j-i)
% =0.5 (i-1)i + 0.5 (k-i)(k-i-1)
% =0.5 [i^2 - i + k^2 + i^2 - 2ki -k +i]
% =i^2 - ki + k(k-1)/2
As a result, $Tr_G(w_i)$ is a quadratic function in~$i$ with a positive coefficient of $i^2$ 
and the $x$-coordinate of its vertex at $x=\frac{k+|C_v|-|C_u|}2$.
Then $Tr_G(w_{i-1})\geq Tr_G(w_i)$ for $i\leq x$
and $Tr_G(w_i)\leq Tr_G(w_{i+1})$ for $x\leq i$,
so that the sequence $Tr_G(w_0),\dots,Tr_G(w_k)$ is inversely unimodal
with the minimum element equal to the smaller one
between $Tr_G(w_{\lfloor x\rfloor})$ and $Tr_G(w_{\lceil x\rceil})$.

\begin{figure}[ht!]
\begin{center}
\includegraphics[width=0.4\textwidth]{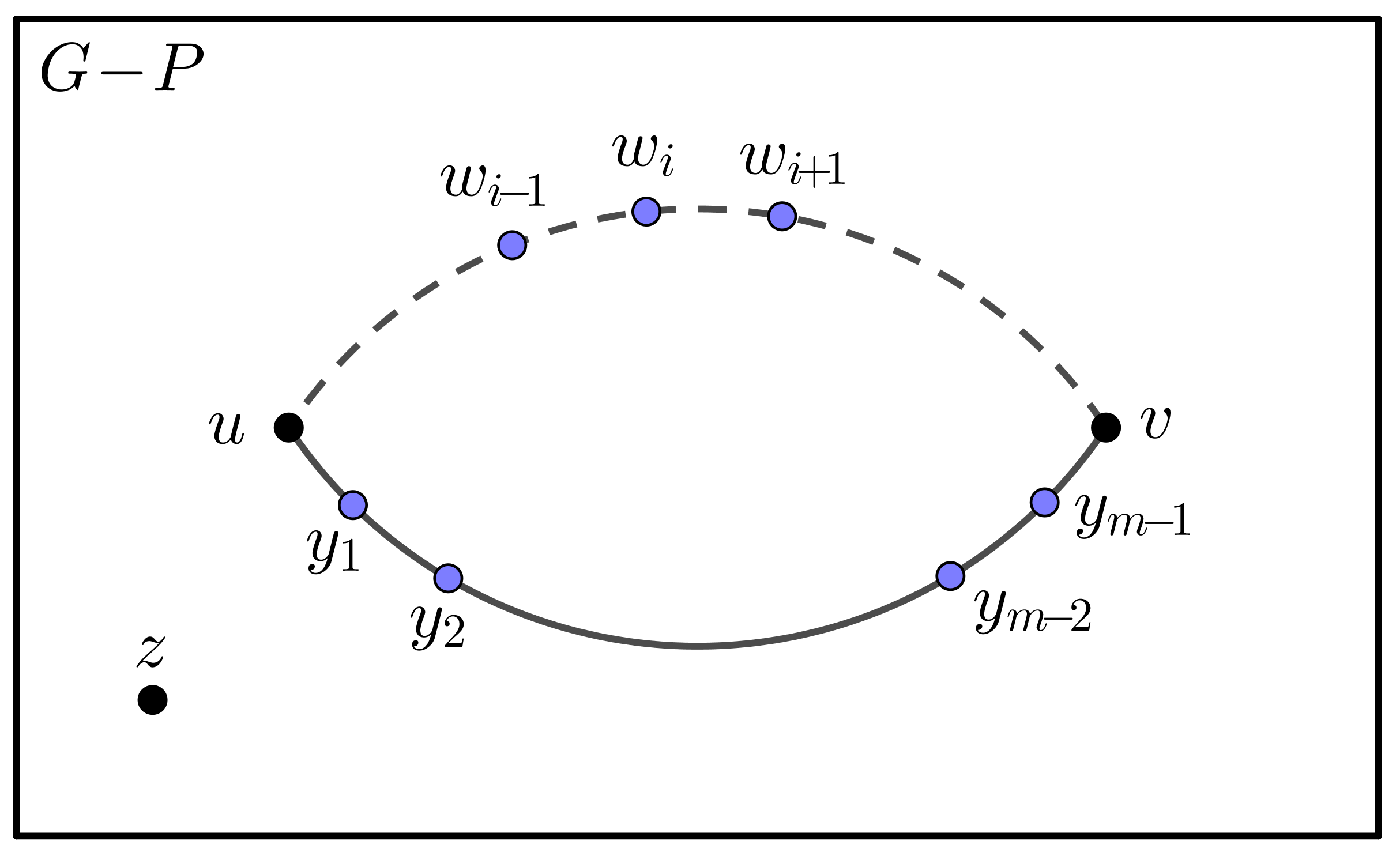}
\end{center}
\caption{The case of connected $H=G-w_1-\dots-w_{k-1}$.}
\label{fig-unimod-2}
\end{figure}

Next, suppose that $H$ is connected and 
let $u,y_1,\dots,y_{m-1},v$ be the shortest walk between $u$ and~$v$ in~$H$
(see Fig.~\ref{fig-unimod-2}).
As $w_1,\dots,w_{k-1}$ do not belong to~$H$,
the set of vertices 
$$
C=\{u,w_1,\dots,w_{k-1},v,y_{m-1},\dots,y_1\}
$$ 
induces a chordless cycle in~$G$.
Hence for each $0\leq i\leq k$ 
the sum $\sum_{z\in C} d_G(w_i,z)$ is equal to 
the transmission of any vertex of the cycle on $k+m$ vertices,
% r=k+m+1
% r=2q is even
% transmission is 2(1+...+(q-1))+q = 2(q-1)q/2+q=q^2
% r=2q+1 is odd
% transmission is 2(1+...+q)=2q(q+1)/2=q(q+1)
% these formulas are unified by \lfloor r/2\rfloor \lceil r/2\rceil
i.e.,
\begin{equation}
\label{eq-sum-C}
\sum_{z\in C} d_G(w_i,z) = \floorfrac{k+m}2\ceilfrac{k+m}2.
\end{equation}

For any other vertex $z\notin C$ and any $0<i<k$,
the shortest walk from $w_i$ to~$z$ in~$G$ goes either through $w_{i-1}$ or $w_{i+1}$.
In the former case,
$$
d_G(w_i,z) = d_G(w_{i-1},z) + 1
\quad\mbox{and}\quad
d_G(w_{i+1},z) \leq d_G(w_i,z) + 1,
$$
while in the latter case,
$$
d_G(w_i,z) = d_G(w_{i+1},z) + 1
\quad\mbox{and}\quad
d_G(w_{i-1},z) \leq d_G(w_i,z) + 1.
$$
In either case, we have
$$
2d_G(w_i,z) \geq d_G(w_{i-1},z) + d_G(w_{i+1},z).
$$
Summing this inequality over all $z\notin C$ 
and taking into account Eq.~(\ref{eq-sum-C}),
we obtain
$$
2Tr_G(w_i) \geq Tr_G(w_{i-1}) + Tr_G(w_{i+1}).
$$
% Tr_G(w_i)-Tr_G(w_{i-1}) \geq Tr_G(w_{i+1}) - Tr_G(w_i).
This implies that the sequence of differences given by 
$$
Tr_G(w_1)-Tr_G(w_0),\quad Tr_G(w_2)-Tr_G(w_1),\quad\dots,\quad Tr_G(w_k)-Tr_G(w_{k-1})
$$
is nonincreasing,
so that the sequence of transmissions $Tr_G(w_0), Tr_G(w_1), \dots, Tr_G(w_k)$ is unimodal,
with the index of its maximum element equal to
the largest~$i$ for which the difference $Tr_G(w_i)-Tr_G(w_{i-1})$ is nonnegative.
\end{proof}

The only related result that we could find in the literature is~\cite[Theorem 3.3]{enjs},
which states that transmissions in a tree~$T$ increase along any path 
that starts from the vertex of minimum transmission in~$T$.

\section{Cartesian product and modulo transmission irregular graphs}
\label{sc-cartesian-product}

Vertex transmissions in Cartesian product of graphs can be expressed 
in terms of vertex transmissions in its factors.
Recall that the {\em Cartesian product} $G\square H$ of two graphs $G=(V_G,E_G)$ and $H=(V_H,E_H)$ is
a graph with the vertex set $V_G\times V_H$ in which
two vertices $(u_G, u_H)$ and $(v_G, v_H)$ are adjacent 
if either $(u_G=v_G$ and $(u_H, v_H)\in E_H)$ or $((u_G,v_G)\in E_G$ and $u_H=v_H)$.
As each edge in a walk between two vertices in~$G\square H$
makes a step in one of the coordinates and keeps the other coordinates fixed,
it is apparent that
$$
d_{G\square H}((u_G,u_H), (v_G,v_H)) = d_G(u_G,v_G) + d_H(u_H,v_H).
$$
Summing over all vertices $(v_G,v_H)$ in $G\square H$, we obtain
\begin{eqnarray}
\nonumber
Tr_{G\square H}(u_G,u_H) 
  &=& \sum_{(v_G,v_H)\in V_{G\square H}} d_{G\square H}((u_G,u_H),(v_G,v_H)) \\
\nonumber  
  &=& \sum_{v_G\in V_G}\sum_{v_H\in V_H} d_G(u_G,v_G) + d_H(u_H, v_H) \\
\nonumber  
  &=& \sum_{v_G\in V_G} \left(|V_H|d_G(u_G,v_G) + Tr_H(u_H)\right) \\
\label{eq-cartesian-transmission}  
  &=& |V_H|Tr_G(u_G) + |V_G|Tr_H(u_H).
\end{eqnarray}
If we now assume that two vertices in the Cartesian product have equal transmissions,
$$
Tr_{G\square H}(u_G,u_H) = Tr_{G\square H}(v_G,v_H),
$$
then Eq.~(\ref{eq-cartesian-transmission}) implies
$$
|V_H|\left(Tr_G(u_G) - Tr_G(v_G)\right) = |V_G|\left(Tr_H(v_H) - Tr_H(u_H)\right).
$$
If $|V_G|$ and $|V_H|$ are further assumed to be relatively prime,
then this implies
$$
Tr_G(u_G) \equiv Tr_G(v_G) \pmod{|V_G|}
\quad\mbox{and}\quad
Tr_H(u_H) \equiv Tr_H(v_H) \pmod{|V_H|}.
$$
This warrants the introduction of the following definition.
\begin{definition}
A connected graph $G$ is {\em modulo transmission irregular} (MTI)
if no two vertices of~$G$ have transmissions congruent modulo the number of vertices of~$G$.
\end{definition}
The above argument then gives rise to the following theorem.
\begin{theorem}
\label{th-product-with-mti}
Let $G$ be a modulo transmission irregular graph,
and let $H$ be a connected graph.
If $|V_G|$ and $|V_H|$ are relatively prime, then:

i) if $H$ is transmission irregular,
   then $G\square H$ is also transmission irregular;

ii) if $H$ is modulo transmission irregular,
    then $G\square H$ is also modulo transmission irregular.   
\end{theorem}
  
\begin{proof}
We will prove part ii) only. The part i) follows analogously. 
Assume therefore that 
$$
Tr_{G\square H}(u_G,u_H) \equiv Tr_{G\square H}(v_G,v_H) \pmod{|V_{G\square H}|}.
$$
From Eq.~(\ref{eq-cartesian-transmission}), we obtain
$$
|V_G||V_H|\,\Big|\,|V_H|\left(Tr_G(u_G) - Tr_G(v_G)\right) + |V_G|\left(Tr_H(v_H) - Tr_H(u_H)\right),
$$
and since $|V_G|$ and $|V_H|$ are relatively prime, we have
$$
Tr_G(u_G) \equiv Tr_G(v_G) \pmod{|V_G|}
\quad\mbox{and}\quad
Tr_H(u_H) \equiv Tr_H(v_H) \pmod{|V_H|}.
$$
Finally, since both $G$ and $H$ are MTI,
this implies that $u_G=v_G$ and $u_H=v_H$,
so that $G\square H$ is also MTI.
\end{proof}  

MTI graphs are situated between TI and ITI graphs:
each ITI graph is also MTI, while each MTI graph is also TI.
On up to ten vertices, each MTI graph is also ITI,
while the complete enumeration of connected graphs on 11~vertices shows that
there are 1072 ITI graphs and 1293 MTI graphs.
Fig.~\ref{fig-MTI-not-ITI} shows one of the examples of MTI graph that is not ITI,
let us denote it by~$M_{11}$.

\begin{figure}[ht!]
\begin{center}
\includegraphics[width=1.25in]{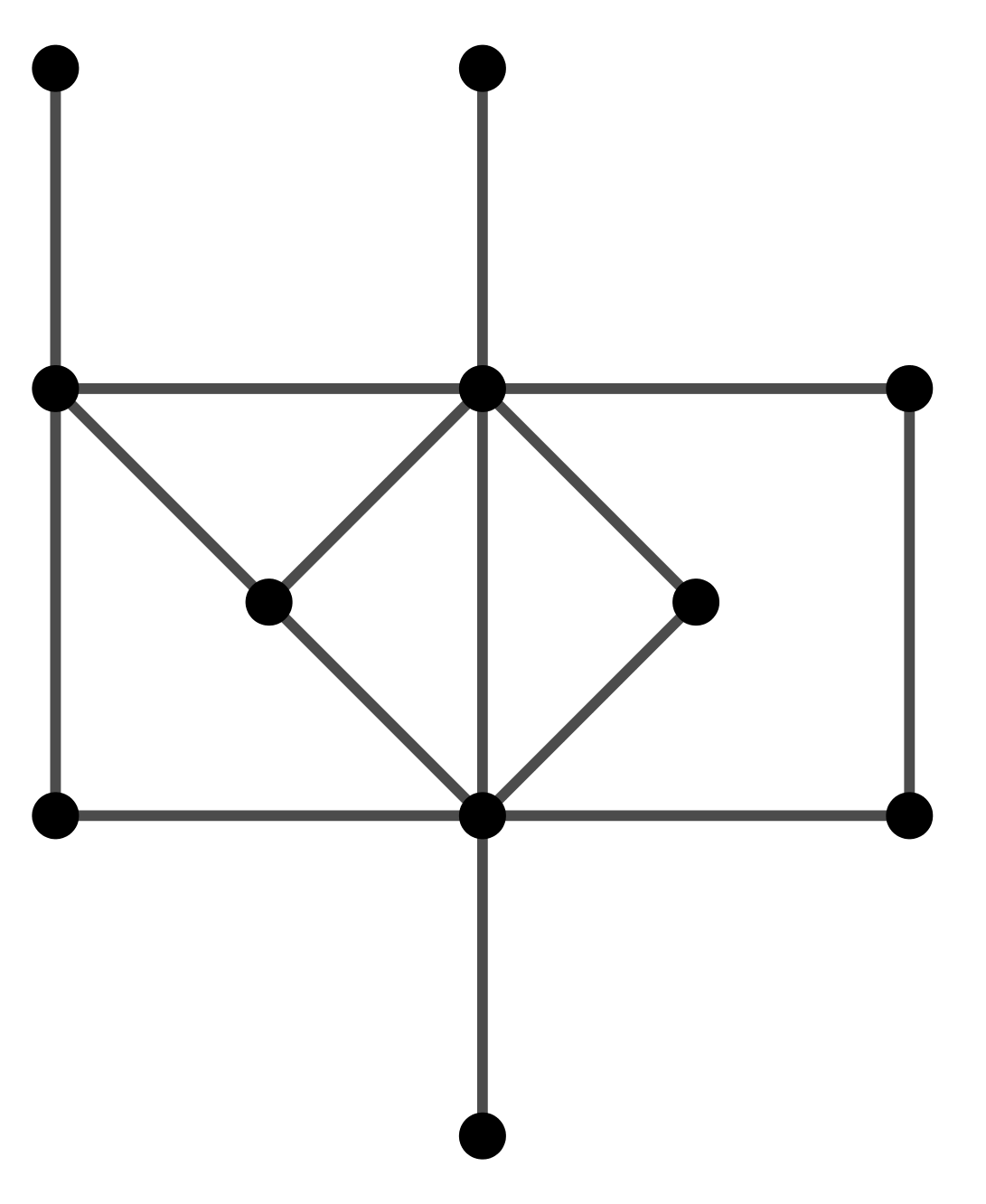}
\end{center}
\caption{An 11-vertex MTI graph that is not ITI. 
The set of its vertex transmissions is $\{14,15,17,\dots,24,27\}$.}
\label{fig-MTI-not-ITI}
\end{figure}

Theorem~\ref{th-product-with-mti} now makes it possible 
to obtain new families of TI graphs from the existing ones.
For example:
\begin{itemize}
\item the graph $\mathcal{G}^1_{6m+1}$ is TI for $m\geq 1$ by Proposition~\ref{pr-1} and 
      it has $6(m+1)$ vertices. Hence if $11\nmid m+1$, then
      $M_{11} \square \mathcal{G}^1_{6m+1}$ is also TI graph with $66(m+1)$ vertices;

\item the graph $\mathcal{G}^2_{4m+1}$ is TI for $m\geq 1$ by Proposition~\ref{pr-2} and
      it has $4m+9$ vertices. Hence if $11\nmid m+5$, then
                                              %(4m+9)+11=4m+20=4(m+5)
      $M_{11} \square \mathcal{G}^2_{4m+1}$ is also TI graph with $11(4m+9)$ vertices;

\item the graph $\mathcal{G}^3_n$ is TI for $n\geq 2$ by Proposition~\ref{pr-3} and
      it has $2n+8$ vertices. Hence if $11\nmid n+4$, then
      $M_{11} \square \mathcal{G}^3_n$ is also TI graph with $22(n+4)$ vertices.
\end{itemize}      

Note, however, that Theorem~\ref{th-product-with-mti} cannot be used to produce 
an infinite family of MTI graphs from a finite collection of existing MTI graphs,
due to the request for relatively prime numbers of vertices among factor graphs.
% Hence another approach is needed for obtaining an infinite family of ITI or MTI graphs.

\section{An infinite family of ITI graphs}
\label{sc-infinite-iti}

After spending considerable efforts on trying 
to construct an infinite family of ITI graphs through addition of chordal paths
motivated by the apparent abundance of Hamiltonian ITI graphs,
we eventually had to look for another approach to positively answer Dobrynin's question.
Because the number of ITI graphs on up to ten vertices is rather small
(all 16 of them are shown in Fig.~\ref{fig-small-ITI}) 
and since there are no ITI graphs on ten vertices,
we proceeded to enumerate ITI graphs among connected graphs on 11 vertices.
The unexpectedly large number of 1072 such graphs compelled us to calculate some of pertinent statistics first.
It turned out that a great majority of them (over 900 graphs) has diameter three only.
Most of these diameter three ITI graphs tend to have 
two adjacent vertices $a$ and~$b$ of large degrees which differ by one.
To further simplify the structure of considered ITI instances,
we required that each remaining vertex should be adjacent to at least one of $a$ and~$b$,
which resulted in 151 such ITI graphs on 11 vertices.
When we further classified these graphs 
according to the numbers of pendent vertices adjacent to $a$ or~$b$,
it turned out that in 114 of them,
each of $a$ and~$b$ is adjacent to a single pendent vertex.
It is only then that we observed 
that a good number of small ITI graphs from Fig.~\ref{fig-small-ITI} also has this structure,
including the smallest ITI graph on 7 vertices.

\begin{figure}[ht!]
\begin{center}
\includegraphics[height=1.6in]{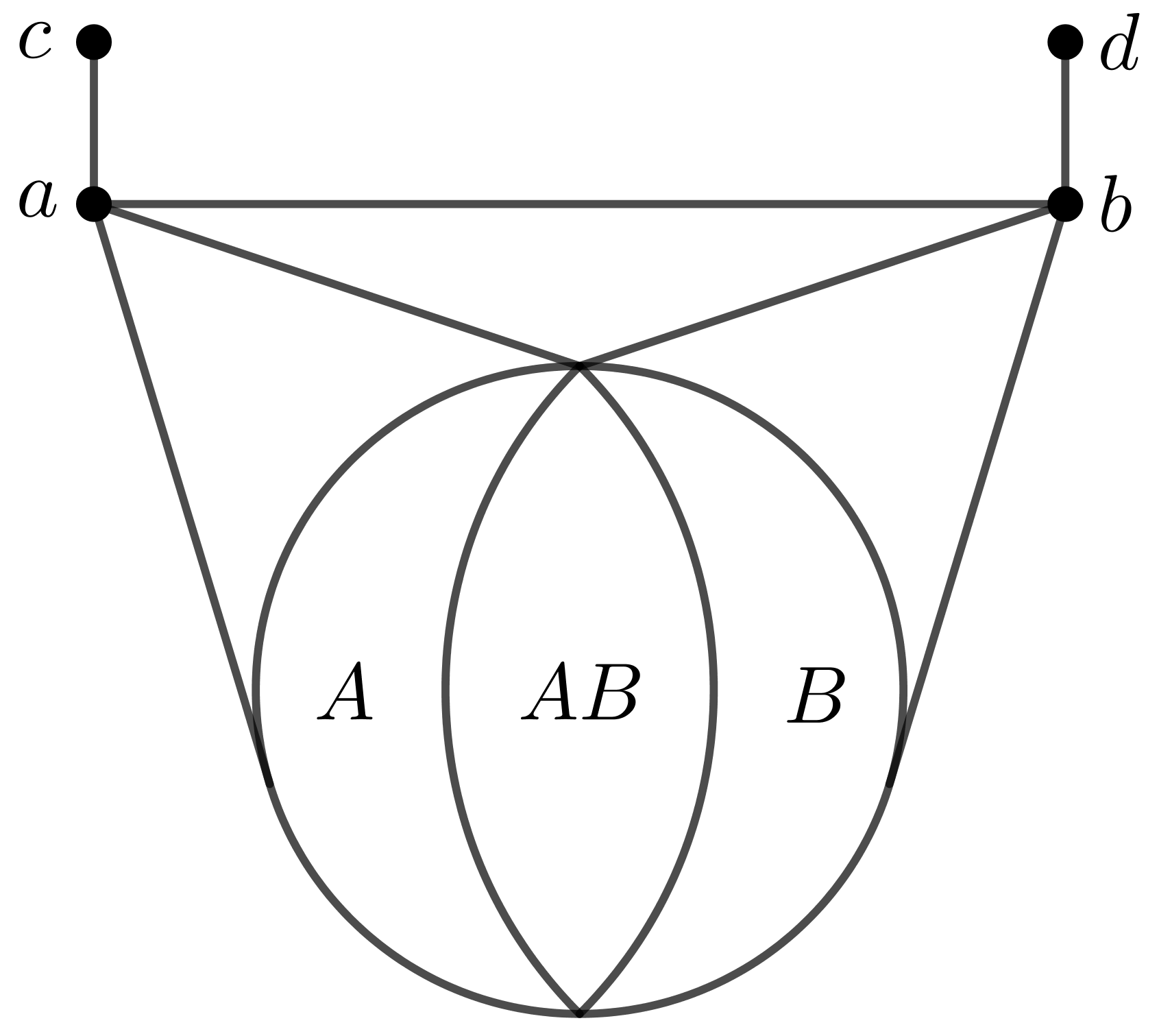}
\end{center}
\caption{The general structure of many instances of ITI graphs.}
\label{fig-ITI-structure}
\end{figure}

Fig.~\ref{fig-ITI-structure} illustrates this observed structure of ITI instances.
Here $A$ denotes the set of vertices adjacent to~$a$ only (excluding the pendent vertex~$c$),
$B$ denotes the set of vertices adjacent to $b$ only (excluding $d$),
while $AB$ is the set of vertices adjacent to both $a$ and~$b$.
In the final selection of ITI instances on 11 vertices,
the set~$AB$ consisted of exactly two vertices,
while $|A|=|B|+1$, 
due to the requirement that the degree of~$a$ is one larger than the degree of~$b$.

To make sure that this setup would lead to larger examples of ITI graphs,
we wrote a small program that partitioned the sets of all graphs on~9 and on 11~vertices
as $A\cup AB\cup B$ in all possible ways satisfying $|A|=|B|+1$ and $|AB|=2$,
and then added the new vertices $a,b,c$ and $d$ in the corresponding manner.
Among graphs on 9~vertices,
the program found many examples that in this way led to larger ITI graphs
within the first minute of execution,
while on 11~vertices
the program found more than 14,000 ways to build larger ITI graphs
from the first 505,000 graphs with 11 vertices processed during a single night.
This was enough to suggest that this setup can really lead to an infinite ITI family.

Assume now that $|A|=k$ (so that $|B|=k-1$) for some $k\geq 1$.
Transmissions of the vertices $a,b,c,d$ under this general structure are as follows:
\begin{alignat*}{2}
%\label{eq-tr-a}
Tr(a) &= |A|+|AB|+2|B|+4   &&= 3k+4, \\
%\label{eq-tr-b}
Tr(b) &= 2|A|+|AB|+|B|+4   &&= 3k+5, \\
%\label{eq-tr-c}
Tr(c) &= 2|A|+2|AB|+3|B|+6 &&= 5k+7, \\
%\label{eq-tr-d}
Tr(d) &= 3|A|+2|AB|+2|B|+6 &&= 5k+8, \\
\end{alignat*}
while for $u\in A$, $v\in B$ and $w\in AB$ we have, respectively:
\begin{eqnarray}
\nonumber
Tr(u) 
&=& [\deg_A(u) + 2(|A|-1-\deg_A(u))] \\
\nonumber
&+& [\deg_{AB}(u) + 2(|AB|-\deg_{AB}(u))] \\
\nonumber
&+& [\deg_B(u) + 2\deg^{(2)}_B(u) + 3(|B|-\deg_B(u)-\deg^{(2)}_B(u))] + 8 \\
\label{eq-tr-uinA}
&=& 5k+7 - [\deg_A(u) + \deg_{AB}(u) + 2\deg_B(u) + \deg^{(2)}_B(u)],
\end{eqnarray}
\begin{eqnarray}
\nonumber
Tr(v)
&=& [\deg_A(v) + 2\deg^{(2)}_A(u) + 3(|A|-\deg_A(v)-\deg^{(2)}_A(v))] \\
\nonumber
&+& [\deg_{AB}(v) + 2(|AB|-\deg_{AB}(v))] \\
\nonumber
&+& [\deg_B(v) + 2(|B|-1-deg_B(v))] + 8 \\
\label{eq-tr-vinB}
&=& 5k+8 - [2\deg_A(v) + \deg^{(2)}_A(v) + \deg_{AB}(v) + \deg_B(v)],
\end{eqnarray}
and
\begin{eqnarray}
\nonumber
Tr(w)
&=& [\deg_A(w) + 2(|A|-\deg_A(w))] \\
\nonumber
&+& [\deg_{AB}(w) + 2(|AB|-1-deg_{AB}(w))] \\
\nonumber
&+& [\deg_B(w) + 2(|B|-\deg_B(w))] + 6 \\
\label{eq-tr-winAB}
&=& 4k+6 - [\deg_A(w) + \deg_{AB}(w) + \deg_B(w)],
\end{eqnarray}
where $\deg_X(y)$ denotes the number of vertices of~$X$ adjacent to~$y$,
while $\deg^{(2)}_X(y)$ denotes the number of vertices of~$X$ at distance two from~$y$
in the subgraph induced by $A\cup AB\cup B$.

\begin{figure}[ht!]
\begin{center}
\includegraphics[height=1.35in]{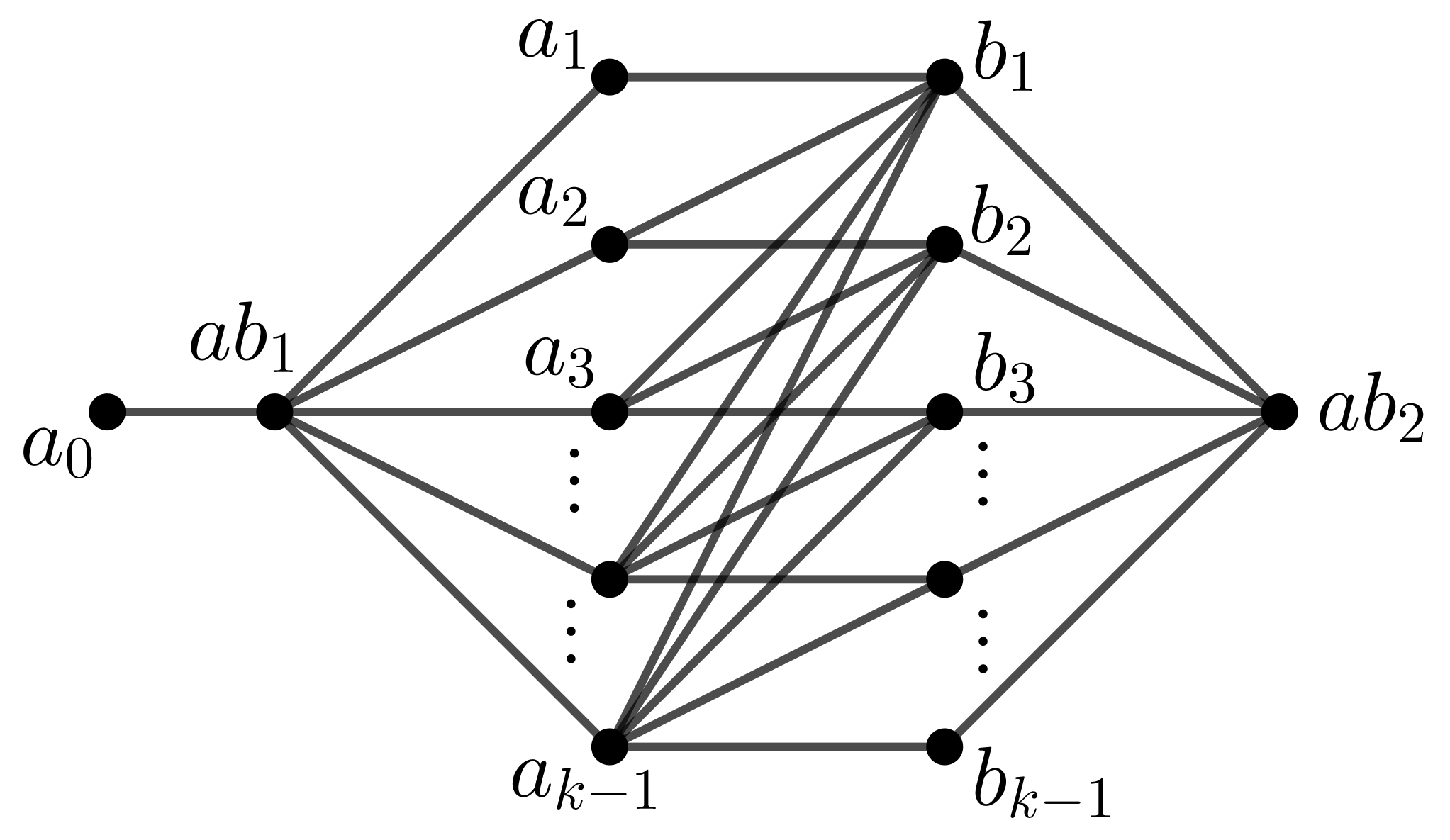}
\end{center}
\caption{Particular structure of the subgraph induced by $A\cup AB\cup B$.}
\label{fig-ITI-AuABuB-subgraph}
\end{figure}

Let us now introduce a particular structure for the subgraph induced by $A\cup AB\cup B$,
as illustrated in Fig.~\ref{fig-ITI-AuABuB-subgraph}.
The set $A$ is made up of vertices $a_0,a_1,\dots,a_{k-1}$,
the set $B$ of vertices $b_1,\dots,b_{k-1}$,
and the set $AB$ contains vertices $ab_1$ and~$ab_2$.
Adjacencies are set so that
the vertex $a_i$ is adjacent to the vertices $b_1,\dots,b_i$ for each $i=0,\dots,k-1$,
hence $\deg_B(a_i)=i$ and $\deg_A(b_j)=k-j$ for each feasible $i$ and~$j$.
With this structure, we have
\begin{alignat*}{2}
& \deg_A(a_i) + \deg_{AB}(a_i) + 2\deg_B(a_i) + \deg^{(2)}_B(a_i) &&= 2i+1, \\
& 2\deg_A(b_j) + \deg^{(2)}_A(b_j) + \deg_{AB}(b_j) + \deg_B(b_j) &&= 2(k-j)+1, \\
& \deg_A(ab_1) + \deg_{AB}(ab_1) + \deg_B(ab_1) &&= k, \\
& \deg_A(ab_2) + \deg_{AB}(ab_2) + \deg_B(ab_2) &&= k-1.
\end{alignat*}
Hence, from Eqs.~(\ref{eq-tr-uinA})--(\ref{eq-tr-winAB}) it follows that
\begin{alignat*}{2}
Tr(a_i) &= 5k+6-2i \qquad\mbox{for $i=0,\dots,k-1$,} \\
Tr(b_j) &= 3k+7+2j \qquad\mbox{for $j=1,\dots,k-1$,} \\
Tr(ab_1) &= 3k+6, \\
Tr(ab_2) &= 3k+7.
\end{alignat*}
Hence, the transmissions of the vertices $a$, $b$, $ab_1$ and~$ab_2$ 
are equal to $3k+4$, $3k+5$, $3k+6$ and $3k+7$, respectively.
Transmissions of the vertices $b_1,\dots,b_{k-1}$ 
take every second value from $3k+9$ to~$5k+5$,
transmissions of the vertices $a_0,a_1,\dots,a_{k-1}$
take every second value from $3k+8$ to~$5k+6$,
while transmissions of the vertices $c$ and~$d$
are equal to $5k+7$ and~$5k+8$, respectively.
Thus,
the transmissions are all distinct and form the interval $[3k+4,5k+8]\cap\mathbb{Z}$.
To honor the proposer of this question,
we decided to name the graph just constructed as the {\em Dobrynin graph} $\mathop{Dob}_k$.
We have therefore proved our final result here.
\begin{theorem}
The Dobrynin graph $\mathop{Dob}_k$ is interval transmission irregular for each $k\geq 1$.
\end{theorem}

Fig.~\ref{fig-Dobrynin} shows the Dobrynin graphs for the first few values of~$k$.
Although drawn differently,
note that $\mathop{Dob}_1$ is also the first graph from left in the first row,
while $\mathop{Dob}_2$ is the first graph from right in the second row of Fig.~\ref{fig-small-ITI}.

\begin{figure}[ht!]
\begin{center}
\includegraphics[width=0.65\textwidth]{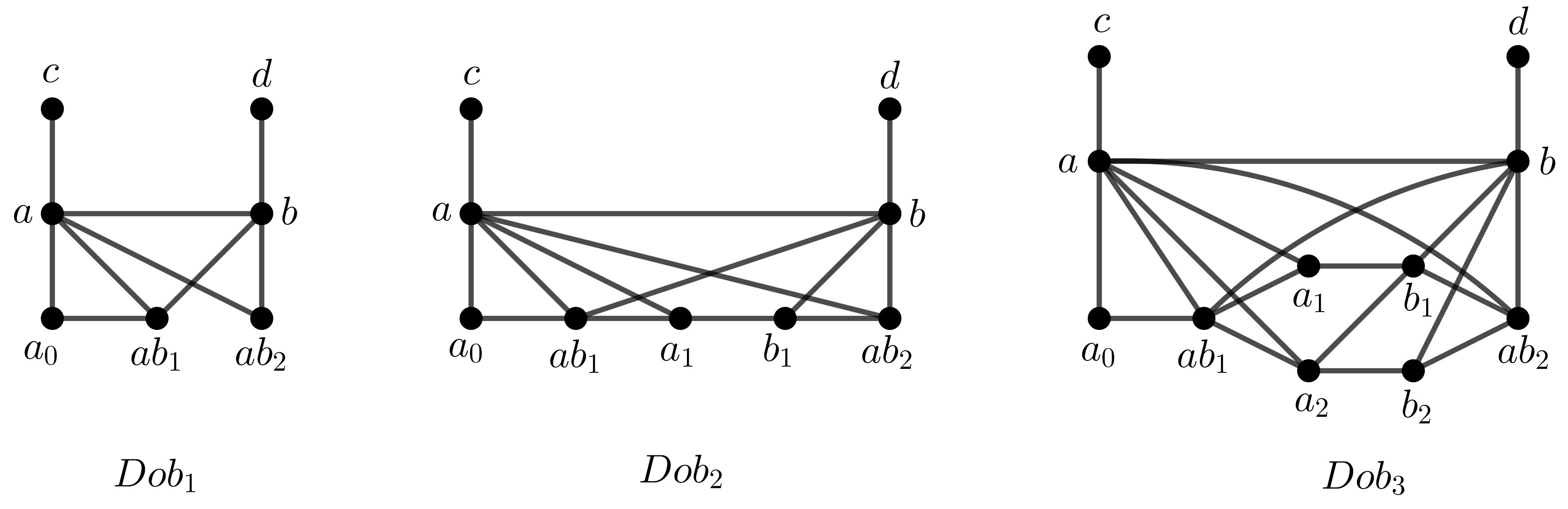}
\end{center}
\caption{Dobrynin graphs for $1\leq k\leq 3$.}
\label{fig-Dobrynin}
\end{figure}

\section{Conclusion}

At the end, 
while ``a picture is worth thousand words''
we have seen here that the drawings in Fig.~\ref{fig-small-ITI} were actually a bit misleading,
as the structure of two relatively dense Dobrynin graphs was not easily recognizable from the drawings alone.
This structure was recognized only 
when a number of descriptive statistics was calculated for a relatively large number of ITI graphs on 11 vertices.
Nevertheless, the ITI graph with the simplest structure in Fig.~\ref{fig-small-ITI} 
did lead to the discovery of results presented in Section~\ref{sc-chordal-paths}.

On the other hand,
the existence of an infinite ITI family also serves to at least partially justify an abundance of TI graphs 
that we observed in our computational studies.
Thanks to Theorem~\ref{th-product-with-mti},
whenever $G$ is an arbitrary (modulo) transmission irregular graph
then the family of Cartesian products $G\square\mathop{Dob}_k$, for all $k$ relatively prime to~$|V_G|$,
forms an infinite family of (modulo) transmission irregular graphs,
making further search for infinite families of TI graphs somewhat redundant,
unless very specific conditions are put for such a family,
such as those posed by Dobrynin~\cite{dobr2} in his question that we answered here.

Finally, while working on the Dobrynin's question we have seen a multitude of ITI graphs
on various numbers of vertices from 7 all the way up to 25,
but interestingly we have not come across a single ITI graph on 10, 14, 18 or 22 vertices.
We are thus tempted to conjecture that 
{\em there does not exist any interval transmission irregular graph on $4k+2$ vertices for $k\in\mathbb{N}$}.

\section*{Acknowledgements}

This work was supported and funded by Kuwait University Research Grant No. SM04/19.

\end{document}